\documentclass{amsart}

\usepackage{amssymb, amscd, framed}

\newtheorem{Theorem}{Theorem}[section] 
\newtheorem{Lemma}[Theorem]{Lemma} 
\newtheorem{Corollary}[Theorem]{Corollary}
\newtheorem{Proposition}[Theorem]{Proposition}
 
\newtheorem{Definition}[Theorem]{Definition}
\newtheorem{Example}[Theorem]{Example}

\newcommand{\ov}[1]{\overline{#1}}
\newcommand{\op}{\operatorname}

\title[Symmetric bilinear forms and vertices]{Symmetric bilinear forms and vertices in characteristic $2$}
\author{John C. Murray}
\address{
   Maynooth University\\
	 Department of Mathematics \& Statistics\\
         National University of Ireland\\
	 Maynooth, Co.~Kildare, Ireland}
\email{John.Murray@maths.nuim.ie}
\date{December 11, 2015, Revisions March 14, 2016}
\dedicatory{In memory of\/ J. A. Green}

\begin{document}

\begin{abstract}
Let $G$ be a finite group and let $k$ be an algebraically closed field of characteristic $2$. Suppose that $M$ is an indecomposable $kG$-module which affords a non-degenerate $G$-invariant symmetric bilinear form. We assign to $M$ a collection of $2$-subgroups of $G$ called its symmetric vertices, each of which contains a Green vertex of $M$ with index at most $2$.  If $M$ is irreducible then its symmetric vertices are uniquely determined, up to $G$-conjugacy.

If $B$ is the real $2$-block of $G$ containing $M$, we show that each symmetric vertex of $M$ is contained in an extended defect group of $B$. Moreover, we characterise the extended defect groups in terms of symmetric vertices.

In order to prove these results, we develop the theory of involutary $G$-algebras. This allows us to translate questions about symmetric $kG$-modules into questions about projective modules of quadratic type.
\end{abstract}

\maketitle

\section{Introduction and statement of main results}

Let $G$ be a finite group and let $k$ be an algebraically closed field of characteristic $2$. By a $G$-form on a $kG$-module $M$ we mean a non-degenerate $G$-invariant symmetric or symplectic bilinear form $B:M\times M\rightarrow k$. We then refer to $(M,B)$ as a symmetric or symplectic $kG$-module and say that $M$ has symmetric or symplectic type. We say that $M$ has quadratic type if it affords a symplectic $G$-form that is the polarization of a $G$-invariant quadratic form. Note that each symplectic form is symmetric, as $\op{char}(k)=2$, and each indecomposable $kG$-module of symmetric (or symplectic) type belongs to a real $2$-block of $G$ as it is self-dual.

{W.\hspace{.1cm}Willems} noted that the Krull-Schmidt theorem fails for symmetric $kG$-modules \cite[3.13]{Willems}. This paper is our attempt to understand induction and restriction of symmetric modules in the absence of this fundamental tool. Our main idea is to use the Puig correspondence theorem \cite[19.1]{Thevenaz} for involutary $G$-algebras \cite{LandrockManz} to reduce the study of symmetric modules to symmetric projective modules. The latter can be analysed using the methods of \cite{GowWillemsPIMs}.

Recall that $M$ is $H$-projective, for $H\leq G$, if $M$ is a component of an induced module $\op{Ind}_H^GL$, for some $kH$-module $L$. If $M$ is indecomposable, a vertex of $M$ is a subgroup $V$ of $G$ which is minimal subject to $M$ being $V$-projective. A $V$-source of $M$ is then a $kV$-module $Z$ such that $M$ is a component of $\op{Ind}_V^GZ$. J. A. Green \cite{GreenIndecomposable} proved that $V$ is unique up to $G$-conjugacy, and $Z$ up to $N_G(V)$-conjugacy.

We extend the notion of $H$-projectivity to symmetric modules, saying that $(M,B)$ or $B$ is $H$-projective if $(M,B)$ is a component of an induced symmetric module ${\op{Ind}}_H^G(L,B_L)$, for some symmetric $kH$-module $(L,B_L)$. Here $\op{Ind}_H^G(L,B_L)$ is $\op{Ind}_H^GL$ endowed with the standard induced $G$-form $B_L^G$. We say that $M$ is symmetrically $H$-projective if it affords a $H$-projective symmetric form. If $M$ is indecomposable we define a {\em symmetric vertex} of $M$ to be a subgroup $T$ of $G$ that is minimal subject to $M$ being symmetrically $T$-projective.

The standard proof of the uniqueness of Green-vertices uses Mackey's formula for the restriction of induced modules and the Krull-Schmidt theorem. An alternative uses Mackey's formula for products of relative traces in endomorphism rings. Both approaches fail for symmetric vertices, the latter because the product of two self-adjoint endomorphisms is not self-adjoint. Indeed, it is not hard to find a module which has two conjugacy classes of symmetric vertices:

\begin{Example}\label{Ex:D12} The dihedral group $D_{12}$ has two subgroups $H_1$ and $H_2$ that are isomorphic to $S_3$. Let $T_i$ be a Sylow $2$-subgroup of $H_i$ and let $M_i$ be the unique non-trivial irreducible $kH_i$-module, for $i=1,2$. Then $\op{Ind}_{H_1}^{D_{12}}M_1=\op{Ind}_{H_2}^{D_{12}}M_2$ is a projective indecomposable $kD_{12}$-module $P$. It is easy to see that $T_1$ and $T_2$ are symmetric vertices of $P$, but $T_1$ is not conjugate to $T_2$ in $D_{12}$.
\end{Example}

On the other hand, symmetric vertices are closely related to Green vertices:

\begin{Theorem}\label{T:vertex}
Let $M$ be an indecomposable $kG$-module which has symmetric type and let $T$ be a symmetric vertex of\/ $M$. Then $T$ contains a Green vertex $V$ of\/ $M$ and $[T:V]\leq2$. More precisely, exactly one of the following is true:
\begin{itemize}
 \item[(i)] $M$ has a $G$-form which is non-degenerate on a component of $\op{Res}_V^GM$ that is a source of\/ $M$. Then $T=V$ and $M$ is in the principal $2$-block of $G$.
 \item[(ii)] The sources of\/ $M$ are self-dual but each $G$-form on $M$ is degenerate on any component of $\op{Res}_V^GM$ that is a source of\/ $M$. Then $[T:V]=2$.
 \item[(iii)] The sources of\/ $M$ are not self-dual. Then $[T:V]=2$.
\end{itemize}
\end{Theorem}

In particular the symmetric vertices of an indecomposable $kG$-module are $2$-subgroups of $G$. As regards proofs, we show that $[T:V]\leq2$ in Lemma \ref{L:GreenSymmetric}, and treat (i) in Proposition \ref{P:case1} and Corollary \ref{C:principal}, (ii) in Proposition \ref{P:case2} and (iii) in Proposition \ref{P:case3}. We give examples of all three cases in Section \ref{SS:Examples}.

A given indecomposable module can have infinitely many non-isometric forms. Moreover a generic form on a module need not be projective relative to any symmetric vertex of the module. However the forms induced from a symmetric vertex have the following property:

\begin{Theorem}\label{T:TleqH}
Let $M$ be an indecomposable $kG$-module which has symmetric type, let $T$ be a symmetric vertex of\/ $M$ and let $B$ be a $T$-projective $G$-form on $M$. Then $B$ is $H$-projective, for $H\leq G$, if and only if\/ $T\leq_G H$.
\end{Theorem}

In particular if $M$ has a unique form, up to isometry, then it has a unique symmetric vertex, up to $G$-conjugacy. We prove this theorem in Lemma \ref{L:TleqH}. In Example \ref{Ex:D12}, each $T_1$-projective form on $P$ is not $T_2$-projective, and vice-versa.

The trivial $kG$-module is of symmetric type and its symmetric vertices are just the Sylow $2$-subgroups of $G$. P. Fong \cite{Fong} noted that each non-trivial self-dual irreducible $kG$-module has a unique symmetric $G$-form, up to scalars, and this form is symplectic. So Theorem \ref{T:TleqH} implies:

\begin{Theorem}\label{T:vertexIrreducible}
The symmetric vertices of a self-dual irreducible $kG$-module are uniquely determined up to $G$-conjugacy.
\end{Theorem}

We now discuss applications to real $2$-blocks. In fact this paper was motivated by our attempt to find an analogue of the extended defect group of a real $2$-block for self-dual indecomposable $kG$-modules.

Recall that $kG$ is a $k(G\times G)$-module. Let ${\mathcal B}$ be an indecomposable component of this module. Then ${\mathcal B}$ is a block of $kG$, or a $2$-block of $G$. R. Brauer defined a defect group of ${\mathcal B}$ to be a Sylow $2$-subgroup $D$ of the centralizer of a certain element of $G$. J. A. Green showed that the diagonal subgroup $\Delta D\leq G\times G$ is a vertex of ${\mathcal B}$, as $k(G\times G)$-module.

Now $g\rightarrow g^{-1}$, for $g\in G$, extends to an involutary $k$-algebra anti-automorphism $^o$ of $kG$ called the contragredient map. Assume that ${\mathcal B}$ is invariant under $^o$. Then ${\mathcal B}$ is a real $2$-block of $G$. R. Gow defined an extended defect group of ${\mathcal B}$ to be a Sylow $2$-subgroup $E$ of the extended centralizer of a certain real element of $G$ \cite{GowReal2Blocks}. If ${\mathcal B}$ is the principal block $E=D$ is a Sylow $2$-subgroup of $G$. Otherwise $D\leq E$ and $[E:D]=2$.

In \cite{MurrayStrong} we used $E$ to determine a vertex of ${\mathcal B}$, regarded as a module for the wreath product $G\wr C_2$ of $G$ with a cyclic group of order $2$. Now let $B_1$ be the $G\times G$-form on $kG$ with respect to which the elements of $G$ form an orthonormal basis.

\begin{Theorem}\label{T:vertexBlock}
Let ${\mathcal B}$ be a real $2$-block of $G$, with extended defect group $E$. Then
\begin{itemize}
 \item[(i)] Each symmetric vertex of an indecomposable ${\mathcal B}$-module is $G$-conjugate to a subgroup of $E$.
 \item[(ii)] Some self-dual irreducible ${\mathcal B}$-module has symmetric vertex $E$.
 \item[(iii)] The restriction of $B_1$ to ${\mathcal B}$ is non-degenerate and $\Delta E$-projective.\\So $\Delta E$ is a symmetric vertex of\/ ${\mathcal B}$, regarded as a $G\times G$-module.
\end{itemize}
\end{Theorem}

We prove (i) in Lemma \ref{L:part(i)}, (ii) in Corollary \ref{C:part(ii)} and (iii) in Lemma \ref{L:part(iii)}. Note that (i) and (ii) give a new proof of Gow's result \cite[2.2]{GowReal2Blocks} that the extended defect groups of a real $2$-block of $G$ are uniquely determined up to $G$-conjugacy. However we do not know if all symmetric vertices of ${\mathcal B}$ are $G\times G$-conjugate to $\Delta E$.

Regarding the rest of the paper, we give examples of the failure of the Krull-Schmidt Theorem for symmetric $kG$-modules in \ref{SS:Krull-Schmidt}. The rest of Section \ref{S:General} contains basic results on bilinear forms and involutary $k$-algebras. Lemma \ref{L:idempotent_lifting} is a lifting result for self-adjoint idempotents which is a mild generalization of \cite[1.4]{LandrockManz}. We shall use it extensively in the rest of the paper.

We consider adjoints, endomorphisms and bilinear forms in Section \ref{S:Adjoints}. Lemma \ref{L:PerfectGPairing} gives a bijection between $G$-endomorphisms and perfect $G$-pairings on pairs of submodules. Lemma \ref{L:projection} shows that self adjoint idempotents correspond to orthogonal direct summands. In Lemma \ref{L:GowWillems} we give a new proof of \cite[Proposition]{GowWillemsGreen} on decompositions of symmetric modules. Finally Lemma \ref{L:PGL(M,sigma)} shows that in the presence of a $G$-involution each projective representation lifts to a representation.

In Section \ref{S:Induction} we undertake a detailed exploration of form induction. This notion has appeared in many places, for example \cite{GowWillemsGreen}, \cite{Nebe} and \cite{Pacifici}. D. Higman \cite{Higman} provided a criterion for $H$-projectivity using the relative trace map. We define the notion of $H$-projectivity for symmetric forms in \ref{SS:Induction} and prove a symmetric version of Higman's Criterion in Lemma \ref{L:Higman}.

The deepest results of this paper are contained in Section \ref{S:SymmetricVertices}. In \ref{SS:Projective} we consider symmetric bilinear forms on projective $kG$-modules. These can be studied in some detail because of previous work such as \cite{GowWillemsPIMs} and \cite{MurrayPFS}. For example each projective indecomposable module which has symplectic type is of quadratic type. We give a new criterion for this to occur in Proposition \ref{P:quadraticPIM}, simplified here as:

\begin{Theorem}\label{T:quadraticPIM}
Let $M$ be a non-trivial self-dual irreducible $kG$-module and let $B$ be a symplectic $G$-form on $M$. Then $P(M)$ has quadratic type if and only if there is an involution $t\in G$ such that $B(tm,m)\ne0$, for some $m\in M$.
\end{Theorem}

We interpret this result as follows. Suppose that $P(M)$ is not of quadratic type and let $t\in G$ be an involution. Then the image of $t$ in $E(M)$ is an {\em alternate} linear transformation, in the terminology of \cite{GowLaffey}. Equivalently $M=M_1\dot+ M_2$, where $M_1$ and $M_2$ are isomorphic $k\langle t\rangle$-modules that are totally isotropic with respect to $B$.

We have already indicated that the proof of Theorem \ref{T:vertex} is scattered among Sections \ref{SS:2applications}, \ref{SS:Part1}, \ref{SS:Part2}, and \ref{SS:Part3}, and proofs of Theorems \ref{T:TleqH} and \ref{T:vertexIrreducible} are given in Section \ref{SS:Proofs}. We would like to highlight the following version of Proposition \ref{P:PuigP(k_N)}:

\begin{Theorem}\label{T:Scott}
Let $V$ be a $2$-subgroup of $G$ and let $Z$ be an indecomposable $kV$-module with vertex $V$ which is of symmetric type. Then among the  indecomposable components of\/ $\op{Ind}_V^GZ$ which have vertex $V$, there is a distinguished $kG$-module $M$ characterised by any one of:
\begin{itemize}
 \item[(a)] $M$ has multiplicity one as a component of $\op{Ind}_V^GZ$.
 \item[(b)] $M$ has odd multiplicity as a component of $\op{Ind}_V^GZ$.
 \item[(c)] $Z$ is a $B$-component of $\op{Res}_V^GM$, for some $G$-form $B$ on $M$.
 \item[(d)] Each $V$-projective form on $\op{Ind}_V^GZ$ is non-degenerate on $M$.
\end{itemize}
\end{Theorem}

If $Z$ is the trivial $kV$-module, then $M$ is just the Scott-module with vertex $V$. So this theorem extends the notion of a Scott-module from permutation modules to a much larger class of modules in characteristic $2$. Like a Scott module, $M$ will belong to the principal $2$-block of $G$, but unlike a Scott module, $M$ need not have a trivial submodule nor a trivial quotient module.

Unless stated otherwise, $k$ is an algebraically closed field of characteristic $2$. We relax the hypotheses on $k$ in some sections, as will be indicated. Everywhere $G$ is a finite group of even order. All vector spaces and algebras are finite dimensional and all $kG$-modules are left $kG$-modules. If $M$ is a $k$-vector space then $E(M)$ is its endomorphism ring, and if $M$ is a $kG$-module then $P(M)$ is its projective cover.

\section{Generalities}\label{S:General}

Throughout this section $k$ is a field and $M$ is a $k$-vector space.

\subsection{Failure of the Krull-Schmidt theorem}\label{SS:Krull-Schmidt}

This short subsection demonstrates one of the difficulties we face when dealing with symmetric bilinear forms in characteristic $2$: certainly each symmetric $kG$-module decomposes as an orthogonal sum of indecomposables. However the symmetric modules occurring in such a sum are not uniquely determined up to isometry. In order to get to the point quickly, we defer defining all terms until later in the paper.

If $\op{char}(k)\ne2$, R. Gow observed that each self-dual indecomposable $kG$-module has either symmetric or symplectic type. In \cite[3.5]{Willems} W.~Willems showed that such a module has a unique symmetric or symplectic form, up to isometry. Moreover he proved an analogue of the Krull-Schmidt theorem in \cite[3.11]{Willems}: the decomposition of a symmetric or symplectic $kG$-module into orthogonal indecomposables is unique up to isometry.

For the rest of \ref{SS:Krull-Schmidt} we suppose that $\op{char}(k)=2$. Then if $M$ is an indecomposable $kG$-module it is easy to verify that its paired module $(M^*\oplus M,P)$ (defined in \ref{SS:Endomorphisms} below) is orthogonally indecomposable. We use this to construct an example that shows that the Krull-Schmidt theorem fails for symmetric $kG$-modules:

\begin{Example}
Let $M$ be an indecomposable $kG$-module and let $B$ be a $G$-form on $M$. Then the diagonal submodule of\/ ${(M,B)\perp(M,B)\perp(M,B)}$ is isomorphic to $(M,B)$ and its orthogonal complement is isomorphic to $(M^*\oplus M,P)$. So
$$
(M,B)\perp(M,B)\perp(M,B)\cong(M,B)\perp(M^*\oplus M,P),
$$
yet even the number of indecomposable components is different on both sides.
\end{Example}

We do not give proofs of these assertions, except to note that the two decompositions exist when $G$ is trivial and $M\cong k$. The general case then follows easily from Lemma \ref{L:idempotent_lifting} applied in the context of Lemma \ref{L:GowWillems}.

Now let $B_1$ be the standard $G$-form on the regular module $kG$. For $x\in kG$, set $B_x(y,z):=B_1(yx,z)$, for all $y,z\in kG$. Then $B_x$ is symplectic if and only if $B_1(x,1)=0$, and non-degenerate if and only if $B_1(x,x)\ne0_k$. See Section \ref{SS:Projective} for details. We use these facts in our second example:

\begin{Example}{\cite[3.13]{Willems}}
Let $V_4=\{1,r,s,t\}$ be the Klein $4$-group, and let $M$ be the regular $kV_4$-module. Consider the symplectic $kV_4$-module $(M,B_r)\perp(M,B_s)$. Then for each $\alpha,\beta\in k$ with $\alpha\ne\beta$ we have
$$
(M,B_r)\perp(M,B_s)=(M,B_{\alpha r+\beta s})\perp(M,B_{\beta r+\alpha s}).
$$
But for units $x,y\in kV_4$, $(M,B_x)\cong(M,B_y)$ if and only if $y=\lambda x$ for some $\lambda\in k^\times$. So $(M,B_r)\perp(M,B_s)$ has an infinite number of orthogonal decompositions, no two of which have a common indecomposable component, even up to isometry.
\end{Example}

\subsection{Involutary algebras and idempotent lifting}\label{SS:idempotent_lifting}

Let $A$ be a $k$-algebra such that $A/J(A)$ is split semi-simple. An involution on $A$ is a $k$-linear automorphism $a\rightarrow a^\tau$ of $A$ such that $\tau^2=\op{id}_A$ and $(ab)^\tau=b^\tau a^\tau$ for all $a,b\in A$. Note in particular that $\tau$ fixes each scalar multiple of $1_A$. Following Landrock and Manz \cite{LandrockManz} we call $(A,\tau)$ an involutary $k$-algebra.

Recall that $A$ is a $G$-algebra if there is a homomorphism $G\rightarrow\op{Aut}(A)$, written $a\rightarrow{^g}a$ for $g\in G$ and $a\in A$. If in addition $\tau$ is an involution on $A$ which commutes with the $G$-action, we call $(A,\tau)$ an involutary $G$-algebra. If $H\leq G$, we use $A^H$ to denote the subalgebra of $H$-fixed points in $A$.

We prove a mild generalization of an idempotent lifting result \cite[Lemma 1.4]{LandrockManz}. This will be used frequently, often without explicit reference, later in the paper. If\/ $I$ is a $2$-sided ideal of\/ $A$, set $\ov{A}=A/I$ and $\ov{a}=a+I$ in $\ov{A}$. This notation is ambiguous, but should be clear from the context. 

\begin{Lemma}\label{L:idempotent_lifting}
Let $(A,\tau)$ be an involutary $k$-algebra, let $I$ be a $\tau$-invariant $2$-sided ideal of\/ $A$ and let $a\in A\backslash I$.
\begin{itemize}
 \item[(i)] If $a$ is a primitive $\tau$-invariant idempotent in $A$, then $\ov{a}$ is a primitive $\tau$-invariant idempotent in $A/I$.
 \item[(ii)] If\/ $\ov{a}$ is a $\tau$-invariant idempotent in $\ov{A}$ then there is a $\tau$-invariant idempotent $e\in aAa^\tau$ with $\ov{e}=\ov{a}$ and $e=p(aa^\tau)$ for some polynomial $p$ with $p(0)=0$.
 \item[(iii)] If\/ $\ov{a}$ is a $\tau$-invariant primitive idempotent in $\ov{A}$, then there is a $\tau$-invariant primitive idempotent $f$ in $A$ such that $\ov{f}=\ov{a}$.
\end{itemize}
\begin{proof}
Note that $(\ov{A},\tau)$ is an involutary $k$-algebra, via $\ov{u}^{\,\tau}:=\ov{u^\tau}$, for all $u\in A$. 

Assume the hypothesis (i). Then \cite[(2.3)]{KulshammerLectures} implies that $\ov{a}$ is a primitive idempotent in $A/I$, and it is clear that $\ov{a}^\tau=\ov{a}$.

For (ii), we may assume that $\ov{1}$ and $\ov{a}$ are linearly independent in $\ov{A}$. Set $b=aa^\tau$. Then $b^\tau=b$ and $\ov{b}=\ov{a}\,\ov{a}^\tau=\ov{a}$. We apply idempotent lifting \cite[(3.2)]{KulshammerLectures} to the $k$-algebra $k[b]$ modulo its ideal $k[b]\cap I$. So there is an idempotent $e\in k[b]$ such that $\ov{e}=\ov{b}$. Then $e$ is $\tau$-invariant as $b$ is $\tau$-invariant and $k[b]$ is commutative. Write $e=p(b)$, where $p\in k[x]$. Then $\ov{e}=p(0)\ov{1}+(p(1)\!-\!p(0))\ov{a}$. So $p(0)=0$.

Finally, assume (iii). Then by \cite[(3.10)]{KulshammerLectures} there is a primitive idempotent $c\in A$ such that $\ov{c}=\ov{a}$. Applying (ii) to $c$, there is a $\tau$-invariant idempotent $f\in cAc^\tau$ with $\ov{f}=\ov{c}=\ov{a}$. Then $f$ is primitive in $A$ as $cuc^\tau\rightarrow cuc^\tau c$ defines a $k$-algebra isomorphism of $cAc^\tau$ with the local algebra $cAc$.
\end{proof}
\end{Lemma}

We establish notation for points and multiplicity modules. Let $A^\times$ be the group of multiplicative units in $A$. A point of $A$ is an $A^\times$-conjugacy class $\epsilon$ of primitive idempotents of $A$. There is a unique maximal $2$-sided ideal ${\mathfrak M}_\epsilon$ of $A$ which does not contain any idempotent in $\epsilon$. The multiplicity module of $\epsilon$ is the unique irreducible $A$-module $P_\epsilon$ annihilated by ${\mathfrak M}_\epsilon$ and the multiplicity algebra of $\epsilon$ is the simple quotient algebra $A/{\mathfrak M}_\epsilon$. Notice that $A/{\mathfrak M}_\epsilon$ is isomorphic to the endomorphism ring $E(P_\epsilon)$ of $P_\epsilon$. We identify the projection $\pi_\epsilon:A\rightarrow A/{\mathfrak M}_\epsilon$ with a map $A\rightarrow E(P_\epsilon)$.

Now $\epsilon^\tau$ is also a point of $A$ and ${\mathfrak M}_\epsilon\cap{\mathfrak M}_{\epsilon^\tau}$ is a $\tau$-invariant ideal of\/ $A$. Suppose that $\epsilon^\tau\ne\epsilon$. Then $A/{\mathfrak M}_\epsilon\cap{\mathfrak M}_{\epsilon^\tau}\cong A/{\mathfrak M}_\epsilon\times A/{\mathfrak M}_{\epsilon^\tau}$ is a semi-simple $k$-algebra with involution $\tau$ interchanging the two simple components. We identify the projection ${\pi_{\epsilon,\epsilon^\tau}=\pi_\epsilon\times\pi_{\epsilon^\tau}}$ with a map $A\rightarrow E(P_\epsilon)\times E(P_{\epsilon^\tau})$.

\subsection{Symmetric and symplectic forms}\label{SS:symmetric_symplectic_foms}

A bilinear form on $M$ is a map $B:M\times M\rightarrow k$ which is linear in each variable. So $m\rightarrow B(m,-)$ and $m\rightarrow B(-,m)$ are $k$-linear maps from $M$ to its linear dual $M^*=\op{Hom}_k(M,k)$ and $B$ is non-degenerate if and only if both maps are $k$-isomorphisms. Then $B$ is symmetric if $B(m_1,m_2)=B(m_2,m_1)$ and symplectic if $B(m_1,m_1)=0$ for all $m_1,m_2\in M$. A subspace $M_1$ of $M$ is totally isotropic if $B(m_1,m_2)=0$ for all $m_1,m_2\in M_1$. The space of bilinear forms on $M$ can be naturally identified with ${M^*\otimes M^*}$.

Suppose that $B$ is non-degenerate. Then its adjoint is the $k$-algebra anti-automorphism $f\rightarrow f^\sigma$, for $f\in E(M)$ defined by 
$$
B(m_1,fm_2)=B(f^\sigma m_1,m_2),\quad\mbox{for all $m_1,m_2\in M$.}
$$
So $B$ is symmetric if and only if $\sigma$ is an involution on $E(M)$, in which case $(E(M),\sigma)$ is an involutary $k$-algebra and we call $(M,B)$ a symmetric or symplectic $k$-space.

An isometry is a $k$-linear map between symmetric or symplectic spaces which preserves the forms. An isometry is injective, but not necessarily surjective. Two symmetric or symplectic spaces are isomorphic if there is a surjective isometry between them. Two symmetric forms are isometric if the corresponding symmetric spaces are isomorphic. For example, $B$ is isometric to $\lambda^2 B$, for all $\lambda\in k^\times$.

If $B$ is symplectic, then it is alternating, meaning $B(m,m')=-B(m',m)$. Indeed alternating is the same as symplectic when $\op{char}(k)\ne2$. But when $\op{char}(k)=2$ alternating is the same as symmetric. So each symplectic form is symmetric, but there are symmetric forms with are not symplectic. For this reason we will not use the term alternating bilinear form.

Consider when $k$ is algebraically closed and of characteristic $2$. There is no standard name for a symmetric bilinear form which is not symplectic. If $B$ is such a form, and $B$ is non-degenerate, then $M$ has an orthonormal basis with respect to $B$. So we call $B$ a {\em diagonalizable} form. If instead $B$ is symplectic, $M$ has a symplectic basis with respect to $B$, and $\op{dim}(M)$ is even. This discussion implies that, up to isometry, each $k$-vector space has one diagonalizable form, and an additional symplectic form, if its dimension is even.

Let $\op{GL}(M,B)=\{g\in\op{GL}(M)\mid B(gm_1,gm_2)=B(m_1,m_2),\forall\, m_1,m_2\in M\}$ be the isometry group of a symmetric $k$-space $(M,B)$. Symplectic forms are more significant than diagonalizable forms. Using the methods of \cite{Grove} it is not too hard to show that if $k$ is a perfect field of characteristic $2$ and $n=\op{dim}(M)$ then
\begin{equation}\label{E:GL(M,B)}
\op{GL}(M,B)\!\cong\!
\left\{\begin{array}{lrl}
&\op{Sp}(n,k),&
\mbox{if\/ $B$ is symplectic, $n$ even,}\\
&\op{Sp}(n-1,k),\!&
\mbox{if\/ $B$ is diagonalizable, $n$ odd,}\\
&\hspace{-1.3em}{k^{n-1}:\op{Sp}(n\!-\!2,k)},&\!
\mbox{if\/ $B$ is diagonalizable, $n$ even.}\\
\end{array}\right.
\end{equation}

Now suppose that $M$ is a $kG$-module. Then $B$ is $G$-invariant if $B(gm_1,gm_2)=B(m_1,m_2)$, for all $g\in G$. If in addition $B$ is non-degenerate and symmetric or symplectic, we say that $B$ is a $G$-form on $M$. We also say that $M$ has symmetric or symplectic type. So $(M,B)$ corresponds to a representation $G\rightarrow\op{GL}(M,B)$.

Consider the case where $B$ is a symplectic $G$-form on $M$. Then there is a $\op{dim}(M)$ space of quadratic forms on $M$ which polarize to $B$. If one or more of these forms is $G$-invariant, we say that $M$ has quadratic type.

Recall that $E(M)$ is a $G$-algebra, via ${^g}f(m):=gfg^{-1}(m)$, for $m\in M,f\in E(M)$ and $g\in G$. If $B$ is a $G$-form, $\sigma$  inverts the image of each $g\in G$ in $E(M)$. So
$$
({^g}f)^\sigma=(gfg^{-1})^\sigma=gf^\sigma g^{-1}={^g}(f^\sigma),\quad\mbox{for all $f\in E(M)$.}
$$
This implies that $(E(M),\sigma)$ is an involutary $G$-algebra.

As the Krull-Schmidt theorem fails to hold, we need to be able to distinguish between isomorphic submodules of $M$. A direct summand of $M$ is a submodule $M_1$ of $M$ such that $M=M_1+M_2$ and $M_1\cap M_2=0$, for some submodule $M_2$ of $M$. We express this by writing $M=M_1\dot+M_2$. We say that a $kG$-module $L$ is a component of $M$, and write $L\mid M$, if $L$ is isomorphic to a direct summand of $M$.

Now consider when $B$ is a $G$-form on $M$. If $M_1$ is a submodule of $M$, then its orthogonal complement is the submodule $M_1^\perp:=\{m\in M\mid B(M_1,m)=0\}$. So $B$ is non-degenerate on $M_1$ if and only if $M_1\cap M_1^\perp=0$. When this occurs $M=M_1\dot+M_1^\perp$. Then we call $M_1$ a $B$-direct summand of $M$ and write $M=M_1\perp M_1^\perp$.

We say that a symmetric $kG$-module $(L,B_L)$ is a component of $(M,B)$, and write $(L,B_L)\mid(M,B)$ if there is a $kG$-isometry $(L,B_L)\rightarrow(M,B)$. Note that then $L$ is a component of $M$, as the image of an isometry is a direct summand of $M$.

The next result, which is well-known, implies that each indecomposable $kG$-module which affords a diagonalizable form belongs to the principal $2$-block of $G$.

\begin{Lemma}\label{L:diagonal=trivial}
Let $\op{char}(k)=2$ and suppose that $M$ affords a $G$-form which is not symplectic. Then $M$ has a trivial submodule and a trivial quotient module. 
\begin{proof}
We may assume that $k$ is quadratically closed. Let $B$ be a $G$-invariant symmetric bilinear form on $M$ which is not symplectic. Set ${q(m)=\sqrt{B(m,m)}}$, for $m\in M$. Then ${q:M\rightarrow k}$ is a non-zero $kG$-homomorphism. So $M$ has a trivial quotient module. Suppose that $B$ is non-degenerate. Let $\eta$ be the sum of the vectors in any orthonormal basis for $M$. Then $q(m)=B(\eta,m)$, for all $m\in M$. As a consequence $\eta$ spans a trivial submodule of $M$.
\end{proof}
\end{Lemma}

We will turn the somewhat anomalous nature of diagonalizable $G$-forms to our advantage in Section \ref{S:SymmetricVertices}.

If $\op{char}(k)=0$, each self-dual irreducible $kG$-module has symmetric or symplectic type. The type is detected by the Frobenius-Schur indicator of the associated ordinary irreducible character of $G$. This was generalised to $\op{char}(k)\ne2$ by R.~Gow. He noted that in that case each self-dual indecomposable $kG$-module has either symmetric or symplectic type. The problem of determining the type of a self-dual irreducible $kG$-module $M$ was solved independently by W. Willems \cite[3.11]{Willems} and J.~G.~Thompson \cite{Thompson}. The type is given by the Frobenius-Schur indicator of a real ordinary irreducible character of $G$ which occurs with odd multiplicity in the principal indecomposable character of $G$ corresponding to $M$.

When $\op{char}(k)\ne2$, a symmetric form $B$ can be conflated with the quadratic form $Q(m):=B(m,m)$, as $B(m_1,m_2)=\frac{1}{2}(Q(m_1+m_2)-Q(m_1)-Q(m_2))$ is half the polarization of $Q$. However when $\op{char}(k)=2$, the polarization of a quadratic form is symplectic, and there is a $1$-parameter family of quadratic forms which polarize to a given symplectic form. In particular a $G$-invariant symplectic form need not be the polarization of a $G$-invariant quadratic form. It is an open problem to determine whether a self-dual irreducible $kG$-module has quadratic type in characteristic $2$. The methods in this paper do not appear to have a bearing on this problem.

We note that when $\op{char}(k)=2$, there are self-dual indecomposable $kG$-modules which are neither of symmetric or symplectic type. The unique non-trivial projective indecomposable ${k(C_3\rtimes C_4)}$-module is an example, as follows from  \cite[1.3]{MurrayPFS}.

\section{Adjoints}\label{S:Adjoints}

Throughout this section $k$ is an algebraically closed field, $M$ is a $k$-vector space, $B$ is a non-degenerate symmetric or symplectic bilinear form on $M$ and $\sigma$ is the adjoint of $B$ on $E(M)$.

\subsection{Bilinear forms}\label{SS:BilinearForms}

For each $f\in E(M)$, define the bilinear form $B_f$ on $M$ as
\begin{equation}\label{E:Bf}
B_f(m_1,m_2):=B(fm_1,m_2),\qquad\mbox{for all $m_1,m_2\in M$.}                                                                                                                                                                                                                                                \end  {equation}
Then ${f\rightarrow B_f}$ established a non-canonical $k$-isomorphism ${E(M)\cong M^*\otimes_k M^*}$ and
\begin{itemize}
 \item[] $B_f$ is symmetric $\Longleftrightarrow$ $f=f^\sigma$,
 \item[] $B_f$ is symplectic $\Longleftrightarrow$ $f=f_1-f_1^\sigma$, for some $f_1\in E(M)$,
 \item[] $B_f$ is non-degenerate $\Longleftrightarrow$ $f\in\op{GL}(M)$.
\end{itemize}

The adjoint of a non-degenerate symmetric or symplectic bilinear form is an involution on $E(M)$. Conversely, each involution on $E(M)$ corresponds to a 1-parameter family of such forms on $M$, as we now show:

\begin{Lemma}\label{L:sigma=B}
(i) Let $\tau$ be an involution on $E(M)$. Then up to a non-zero scalar, $M$ has a unique non-degenerate symmetric or symplectic form whose adjoint is $\tau$.\\
(ii) Let $M=M_1\oplus M_2$ and let $\tau$ be an involution on ${E(M_1)\times E(M_2)}$ which interchanges $E(M_1)$ with $E(M_2)$. Then $\tau$ has a unique extension to an involution on $E(M)$. The corresponding forms on $M$ are symplectic and totally isotropic on $M_1$ and $M_2$.
\begin{proof}
The hypothesis (ii) implies that $\op{dim}(M_1)=\op{dim}(M_2)=\frac{1}{2}\op{dim}(M)$. So in case (ii) we can and do choose $B$ so that $B(m_i,m_i')=0$, for all $m_i,m_i'\in M_i$, for $i=1,2$. In particular $B$ is symplectic.

Assume the hypothesis (i). Then $\tau\sigma$ is a $k$-algebra automorphism of $E(M)$. So by the Skolem-Noether theorem there is $g\in\op{GL}(M)$ such that $f^{\tau\sigma}=gfg^{-1}$, for all $f\in E(M)$. So $f^\tau=g^{-\sigma}f^\sigma g^\sigma$. Then $B_g$ is a non-degenerate symmetric form whose adjoint is $\tau$. Moreover $g$, and thus $B_g$, is determined up to a scalar.

Assume the hypothesis (ii). Let $e_1,e_2$ be idempotents in $E(M)$ such that $1_M=e_1+e_2$, $e_i M=M_i$ and $\op{ker}(e_i)=M_{3-i}$. Then $E(M_1)\times E(M_2)$ embeds in $E(M)$ as $e_1E(M)e_1+e_2E(M)e_2$. Now $\tau\sigma$ maps each $E(M_i)$ onto itself and hence restricts to an automorphism on the semi-simple $k$-algebra $E(M_1)\times E(M_2)$. Again by the Skolem-Noether theorem there exists $g_1\in\op{GL}(M_1)$, and $g_2\in\op{GL}(M_2)$ such that $(f_1+f_2)^{\tau\sigma}=(g_1f_1g_1^{-1}+g_2f_2g_2^{-1})$, for all $f_1\in E(M_1)$ and $f_2\in E(M_2)$. Moreover each of $g_1$ and $g_2$ is determined up to a scalar. Applying $\sigma$ to both sides, we get $(f_1+f_2)^\tau=g_2^{-\sigma}f_2^{\sigma}g_2^{\sigma}+g_1^{-\sigma}f_1^{\sigma}g_1^{\sigma}$. Then
$$
f_1+f_2=((f_1+f_2)^\tau)^\tau=g_2^{-\sigma}g_1f_1g_1^{-1}g_2^{\sigma}+g_1^{-\sigma}g_2f_2g_2^{-1}g_1^{\sigma}.
$$
This holds for all $f_1\in E(M_1)$. So there is $\lambda\in k^\times$ such that $g_2^{\sigma}=\lambda g_1$. Thus $g_1^{\sigma}=\lambda^{-1} g_2$. Now replace $g_2$ by $\lambda^{-1}g_2$. Then $g_2^{\sigma}=g_1$ and
$$
(f_1+f_2)^\tau=g_1^{-1}f_2^{\sigma}g_1+g_2^{-1}f_1^{\sigma}g_2,\quad\mbox{for all $f_1\in E(M_1),f_2\in E(M_2)$.}
$$
Note that $g_1+g_2\in\op{GL}(M)$. Then $B_{g_1+g_2}$ is a symplectic form on $M$ whose adjoint extends $\tau$ to $E(M)$. Moreover no other involution on $E(M)$ extends $\tau$.
\end{proof}
\end{Lemma}

\subsection{Endomorphisms}\label{SS:Endomorphisms}

Assume now that $M$ is a $kG$-module and $B$ is $G$-invariant. We use $E_G(M)$ to denote the ring of $kG$-endomorphisms of $M$. So $E_G(M)$ is a unital subalgebra of $E(M)$. The $G$-invariant bilinear forms on $M$ are $B_f$ for $f\in E_G(M)$, where $B_f$ is symmetric if and only if $f=f^\sigma$. The following is very useful:

\begin{Lemma}\label{L:M->Sum}
Suppose that $M$ is a non-degenerate indecomposable component of ${(L_1,B_1)\perp\dots\perp(L_n,B_n)}$, where each $(L_i,B_i)$ is a symmetric $kG$-module. Then $M$ is a non-degenerate component of some $(L_i,B_i)$.
\begin{proof}
By hypothesis there is an isometry $\alpha:(M,B_f)\rightarrow(L_1,B_1)\perp\dots\perp(L_n,B_n)$ for some $\sigma$-invariant unit $f\in E_G(M)$. Now there are $kG$-maps $\alpha_i:M\rightarrow L_i$ with
$$
\alpha(m)=\alpha_1(m)+\dots+\alpha_n(m),\quad\mbox{for all $m\in M$.}
$$
Each $(m_1,m_2)\rightarrow B_i(\alpha_im_1,\alpha_im_2)$ is a $G$-invariant symmetric bilinear form on $M$. So there are $\sigma$-invariant $\gamma_i\in E_G(M)$ such that
$$
B_i(\alpha_im_1,\alpha_im_2)=B_{\gamma_i}(m_1,m_2).
$$
Now $\gamma_1+\dots+\gamma_n=f$, as $\alpha$ is an isometry. But $E_G(M)$ is a local ring, as $M$ is indecomposable. So $\gamma_j$ is a unit, for some $j$. Then $B_{\gamma_j}$ is non-degenerate on $M$ and $\alpha_j:(M,B_{\gamma_j})\rightarrow(L_j,B_j)$ is a $kG$-isometry.
\end{proof}
\end{Lemma}

Our next result should be well-known:

\begin{Lemma}\label{L:M/radM}
Suppose that $\theta\in E(M)$. Then $(\theta M)^\perp=\op{ker}(\theta^\sigma)$. So $B$ is non-degenerate on $\theta M$ if and only if\/ $\theta^\sigma$ restricts to an isomorphism $\theta M\rightarrow\theta^\sigma M$.
\begin{proof}
The first statement holds because $m\in(\theta M)^\perp\Longleftrightarrow B(\theta^\sigma m,m')=0$, for all $m'\in M$.
Then $\theta M\cap(\theta M)^\perp=0$ if and only if\/ $\theta^\sigma$ is injective on $\theta M$. But dim$(\theta M)={\rm dim}(\theta^\sigma M)$. So $\theta^\sigma:\theta M\rightarrow\theta^\sigma M$ is injective if and only if it is surjective.
\end{proof}
\end{Lemma}

Next recall that a pairing between two $k$-vector spaces $L$ and $M$ is a bilinear map $L\times M\rightarrow k$. The pairing is perfect if $\ell\rightarrow B(\ell,-)$ for $\ell\in L$, induces a $k$-isomorphism $L\rightarrow M^*$. The canonical example is the perfect pairing $P:M^*\times M\rightarrow k$ given by $P(f,m)=f(m)$, for $f\in M^*$ and $m\in M$. Then $P$ extends to a non-degenerate symplectic bilinear form on $M^*\oplus M$, also denoted by $P$, which is zero on restriction to each of $M^*$ and $M$. We call $(M^*\oplus M,P)$ the paired module on $M$.

Recall that $M^*$ is a $kG$-module via $(gf)(m):=f(g^{-1}m)$, for $g\in G,f\in M^*$ and $m\in M$. A perfect $G$-pairing between two $kG$-modules $L$ and $M$ is a perfect pairing $P$ for which $P(g\ell,gm)=P(\ell,m)$, for all $\ell\in L, m\in M$ and $g\in G$. Equivalently $\ell\rightarrow B(\ell,-)$ for $\ell\in L$, is a $kG$-module isomorphism $L\cong M^*$. For example $P:M^*\times M\rightarrow k$ is a perfect $G$-pairing.

We now associate to each $G$-endomorphism of $M$ a perfect $G$-pairing on a pair of submodules of $M$. This construction is essential to our characterisation of projectivity of forms in Lemma \ref{L:Higman}.

\begin{Lemma}\label{L:PerfectGPairing}
Given $\theta\in E_G(M)$ define $\hat B_\theta:\theta M\times\theta^\sigma M\rightarrow k$ via
$$
\hat B_\theta(\theta m_1,\theta^\sigma m_2)=B(\theta m_1,m_2),\quad\mbox{for all $m_1,m_2\in M$.}
$$
Then $\hat B_\theta$ is a perfect $G$-pairing, and thus ${\theta^\sigma M\cong(\theta M)^*}$ as $kG$-modules.

Conversely if\/ $P$ is a perfect $G$-pairing between submodules $L_1$ and $L_2$ of\/ $M$ then there is a unique $\psi\in E_G(M)$ such that $L_1=\psi M, L_2=\psi^\sigma M$ and $P=\hat B_\psi$.
\begin{proof}
We note that $\hat B_\theta$ is well-defined as $B(\theta m_1,m_2)=B(m_1,\theta^\sigma m_2)$. Lemma \ref{L:M/radM} implies that
$\hat B_\theta$ is a perfect $G$-pairing between $\theta M$ and $\theta^\sigma M$.

Let $P:L_1\times L_2\rightarrow k$ be a perfect $G$-pairing. Then for each $m\in M$ there is $\psi m\in L_1$ such that $B(m,\ell_2)=P(\psi m,\ell_2)$ for all $\ell_2\in L_2$. Check that $\psi\in E_G(M)$ and $\psi M=L_1$. Likewise there is $\psi'\in E_G(M)$ such that $B(\ell_1,m)=P(\ell_1,\psi'm)$ for all $\ell_1\in L_1$. Now for all $m_1,m_2\in M$
$$
B(m_1,\psi'm_2)=P(\psi m_1,\psi'm_2)=B(\psi m_1,m_2).
$$
It follows that $\psi'=\psi^\sigma$ and $P=\hat B_\psi$.
\end{proof}
\end{Lemma}

Note that $\theta^\sigma=\theta$ if and only if\/ $\theta M=\theta^\sigma M$ and $\hat B_\theta$ is symmetric.

\subsection{Idempotents}\label{SS:Idempotents}

We continue with our assumption that $M$ is a $kG$-module and $B$ is $G$-invariant. Now if $M=M_1\dot+\dots\dot+M_t$ as $kG$-modules, then the projections onto the $M_i$ are pairwise orthogonal idempotents in $E_G(M)$ which sum to $1_M$. In the presence of the $G$-form $B$, we need to consider self-adjoint idempotents:

\begin{Lemma}[Orthogonal Projection]\label{L:projection}
Let $L$ be a $kG$-submodule of $M$. Then $B$ is non-degenerate on $L$ if and only if\/ $L=eM$, for some $\sigma$-invariant idempotent $e\in E(M)$. If\/ $e$ exists it is unique, $G$-invariant and $\op{ker}(e)=L^\perp$.
\begin{proof}
Suppose that $e$ exists. Then $e^\sigma M=e^\sigma eM=eM$. So $B$ is non-degenerate on $L$, by Lemma \ref{L:M/radM}. Moreover ${\rm ker}(e)=(e^\sigma M)^\perp=L^\perp$. This ensures that $e$ is unique, and this forces $e\in E_G(M)$.

Conversely, suppose that $L\cap L^\perp=0$. Let $e$ be projection onto $L$ with kernel $L^\perp$. Then ker${(e^\sigma)=(eM)^\perp=L^\perp}$ and ${e^\sigma M={\rm ker}(e)^\perp=L}$. So $e^\sigma$ is projection onto $L$ with kernel $L^\perp$. We deduce that $e^\sigma=e$.
\end{proof}
\end{Lemma}

Each $G$-invariant form on a direct summand of $M$ extends to $M$, as we show:

\begin{Lemma}\label{L:restrictions}
Let $e\in E_G(M)$ be idempotent. Then each $G$-invariant form on $eM$ is the restriction of a $G$-invariant form\/ $B_\theta$ on $M$, for some $\theta\in e^\sigma E_G(M)e$.
\begin{proof}
Using Lemma \ref{L:PerfectGPairing}, $e^\sigma E(M)e={\rm Hom}(eM,e^\sigma M)\cong(eM)^*\otimes(eM)^*$. Also  $B_\phi$ and $B_{e^\sigma\phi e}$ have the same restrictions to $eM$, for all $\phi\in E_G(M)$.

Let $\hat B$ be a $G$-invariant bilinear form on $eM$. Then $\hat B(e\_,e\_)$ defines a $G$-invariant bilinear form on $M$. So there exists ${\theta\in e^\sigma E_G(M)e}$ with $\hat B(em_1,em_2)=B_\theta(m_1,m_2)$, for all $m_1,m_2\in M$.
\end{proof}
\end{Lemma}

Our next result is required in Proposition \ref{P:case3}. The proof uses ideas from \cite{GowWillemsPIMs}:

\begin{Lemma}\label{L:efinS}
Let $A$ be a semi-simple subalgebra of\/ $E(M)$ and let $f\in A\cap A^\sigma$ be an idempotent such that $fM$ is a submodule of $M$ and $B$ is non-degenerate on $fM$. Let $e$ be orthogonal projection onto $fM$. Then $e\in E_G(M)\cap A$.
\begin{proof}
Lemma \ref{L:M/radM} implies that $f^\sigma fM=f^\sigma M$. So by the Artin-Wedderburn Theorem and the Jacobson Density Lemma  $f^\sigma=f^\sigma fa$ for some $a\in A$. Set $e=faf^\sigma$. Then $e\in A$, by hypothesis on $f$ and 
$$
e^\sigma e=(fa^\sigma f^\sigma)(faf^\sigma)=fa^\sigma(f^\sigma fa)f^\sigma=fa^\sigma(f^\sigma)^2=e^\sigma.
$$
So $e=(e^\sigma e)^\sigma=(e^\sigma e)=e^\sigma$. Then $e^2=e^\sigma e=e$. Moreover $e\in A$.

Now $eM=fM$, as $e=fe$ and $f=ff=f(a^\sigma f^\sigma f)=(fa^\sigma f^\sigma)f=ef$. As $e$ is orthogonal projection onto $fM$, Lemma \ref{L:projection} implies that $e\in E_G(M)$.
\end{proof}
\end{Lemma}

Finally we prove \cite[Proposition]{GowWillemsGreen} using methods which will be developed later:

\begin{Lemma}\label{L:GowWillems}
Suppose that $M=M_1\dot+\dots\dot+M_t$ where each $M_i$ is an indecomposable $kG$-module. Then for each $i$
\begin{itemize}
\item[(i)] $B$ is non-degenerate on $M_i$ or
\item[(ii)] $B$ is non-degenerate on $M_i\dot+M_j$ for some $j\ne i$ with $M_j\cong M_i^*$.
\end{itemize}
\begin{proof}
Write $1_M=e_1+\dots e_t$ where $e_1,\dots,e_t$ are pairwise orthogonal primitive idempotents in $E_G(M)$ with $e_jM=M_j$. Let $\epsilon$ be the point of\/ $E_G(M)$ containing $e_i$. We use $\ov{\hphantom{d}\vphantom{s}}$ for images in $E(P_\epsilon)$.

Suppose first that $M_i\cong M_i^*$. Then ${\epsilon^\sigma=\epsilon}$, using Lemma \ref{L:PerfectGPairing}. So $(E(P_\epsilon),\sigma)$ is an involutary $k$-algebra. Lemma \ref{L:sigma=B} implies that $P_\epsilon$ affords a symmetric form $B_\epsilon$ with adjoint $\sigma$. For all $j$ with $e_j\in\epsilon$, choose  $s_j\in\ov{e_j}P_\epsilon$ with $s_j\ne0$. Then the $s_j$ form a basis of\/ $P_\epsilon$.

Say $B_\epsilon(s_i,s_i)\ne0$. Then $B_\epsilon$ is non-degenerate on $ks_i$. Let $x_i\in \ov{e_i}E(P_\epsilon)$ be orthogonal projection onto $ks_i$, as given by Lemma \ref{L:projection}. By Lemma \ref{L:idempotent_lifting}, there is a $\sigma$-invariant idempotent $f_i\in e_iE_G(M)e_i^\sigma$ such that $\ov{f_i}=x_i$. Then $f_iM\subseteq M_i$ and $M_i$ is indecomposable. So $f_iM=M_i$. It then follows from Lemma \ref{L:projection} that $B$ is non-degenerate on $M_i$.

Now say $B_\epsilon(s_i,s_i)=0$. As $B_\epsilon$ is non-degenerate, we may choose $j\ne i$ such that $B_\epsilon(s_i,s_j)\ne0$. So $B_\epsilon$ is non-degenerate on $ks_i+ks_j$. Let $x_{ij}\in\ov{e_i+e_j}E(P_\epsilon)$ be orthogonal projection onto $ks_i+ks_j$, as given by Lemma \ref{L:projection}. By Lemma \ref{L:idempotent_lifting}, there is a $\sigma$-invariant idempotent $f_{ij}\in(e_i+e_j)E_G(M)$ such that $\ov{f_{ij}}=x_{ij}$. Then ${f_{ij}M=M_i+M_j}$, as  $f_{ij}M\subseteq M_i+M_j$ but $f_{ij}$ is not primitive. So $B$ is non-degenerate on $M_i+M_j$, by Lemma \ref{L:projection}.

Finally, suppose that $M_i\not\cong M_i^*$. Then $\epsilon^\sigma\ne\epsilon$ and $(E(P_\epsilon)\times E(P_{\epsilon^\sigma}),\sigma)$ is an involutary $k$-algebra satisfying the hypothesis of Lemma \ref{L:sigma=B}(ii). Let $B_{\epsilon,\epsilon^\sigma}$ be the corresponding symplectic form on $P_\epsilon\oplus P_{\epsilon^\sigma}$. We proceed as above; there exists $j$ such that $e_j\in\epsilon^\sigma$ and $B_{\epsilon,\epsilon^\sigma}$ is non-degenerate on $\pi_{\epsilon,\epsilon^\sigma}(e_i+e_j)(P_\epsilon\oplus P_{\epsilon^\sigma})$. Then there is a $\sigma$-invariant idempotent $f_{ij}\in(e_i+e_j)E_G(M)$ such that $f_{ij}M=M_i+M_j$. So $B$ is non-degenerate on $M_i+M_j$, and $M_j\cong M_i^*$.
\end{proof}
\end{Lemma}

\subsection{Projective Representations}\label{SS:Representations}

For the rest of Section \ref{S:Adjoints}, $k$ has characteristic $2$ and $M$ is a $k$-vector space. The map $f\rightarrow gfg^{-1}$, for $g\in\op{GL}(M)$, is a $k$-algebra automorphism of $E(M)$. Using this, we can identify $\op{PGL}(M)=\op{GL}(M)/k^\times1_M$ with $\op{Aut}(E(M))$. As $\sigma$ is a $k$-algebra anti-automorphism of $E(M)$, it acts on $\op{PGL}(M)$. We set $\op{PGL}(M,\sigma)$ as the centralizer of $\sigma$ in $\op{PGL}(M)$.

\begin{Lemma}\label{L:PGL(M,sigma)}
Each projective representation $\theta:G\rightarrow\op{PGL}(M,\sigma)$ has a unique lift to a group representation $G\rightarrow\op{GL}(M,B)$.
\begin{proof}
Let $\rho:\op{GL}(M)\rightarrow\op{PGL}(M)$ be the projection, with kernel $k^\times1_M$. If $g\in\op{GL}(M,B)$, then $\rho(g)\in\op{PGL}(M,\sigma)$. Suppose that $\lambda1_M\in\op{GL}(M,B)$, with $\lambda\in k^\times$. Then $(\lambda1_M)^\sigma=\lambda^{-1}1_M$. But $(\lambda1_M)^\sigma=\lambda1_M$, as $\sigma$ is $k$-linear. So $\lambda=\lambda^{-1}$. As $\op{char}(k)=2$, it follows that $\lambda=1_k$. This shows that $\rho$ restricts to an injective map $\hat{\rho}:\op{GL}(M,B)\rightarrow\op{PGL}(M,\sigma)$. We claim that $\hat\rho$ is surjective. To see this, let $a\in\op{PGL}(M,\sigma)$. Then $a=\rho(g)$, for some $g\in\op{GL}(M)$. Now
$$
g^{-\sigma}fg^\sigma=({^a\!}f^\sigma)^\sigma={^a\!}f=gfg^{-1},\quad\mbox{for all $f\in E(M)$.}
$$
So $g^\sigma=\lambda g^{-1}$ for some $\lambda\in k^\times$, by the Skolem-Noether theorem. Then $\sqrt{\lambda^{-1}}g\in\op{GL}(M,B)$ and $\hat\rho(\sqrt{\lambda^{-1}}g)=a$. Our claim follows from this.

For uniqueness, suppose that $X$ and $Y$ are representations $G\rightarrow\op{GL}(M,B)$ lifting $\theta$. Then there is a function ${\gamma:G\rightarrow k^\times}$ such that $Y(g)=\gamma(g)X(g)$, for all $g\in G$. Applying $\sigma$ to both sides we get $Y(g^{-1})=\gamma(g)X(g^{-1})$. But $Y(g^{-1})=Y(g)^{-1}=\gamma(g)^{-1}X(g^{-1})$. Comparing, we see that $\gamma(g)^{-1}=\gamma(g)$. So $\gamma(g)=1_k$ and $X(g)=Y(g)$, for all $g\in G$.
\end{proof}
\end{Lemma}

We recall some results from \cite[Appendix A]{MurrayPFS}. Suppose that $e_1$ and $e_2$ are orthogonal idempotents in $E(M)$ such that $1_M=e_1+e_2$ and $e_1^\sigma=e_2$. Set $M_i=e_iM$. So $M=M_1\oplus M_2$ and $B$ is zero on each of $M_1$ and $M_2$ and is symplectic on $M$. The stabilizer of $\{M_1,M_2\}$ in $\op{GL}(M)$ is the group $\op{GL}(M_1,M_2)=\op{GL}(M_1)\times\op{GL}(M_2):\langle s\rangle$. Here we can and do assume that $s$ is an involution which interchanges $M_1$ and $M_2$.

Let $\op{PGL}(M_1,M_2)\cong\op{PGL}(M_1)\wr S_2$ be the group of automorphisms of the semi-simple $k$-algebra $E(M_1)\times E(M_2)$. Now $\op{GL}(M_1,M_2)$ acts by conjugation on $E(M_1)\times E(M_2)$ and the resulting map $\phi:\op{GL}(M_1,M_2)\rightarrow\op{PGL}(M_1,M_2)$ is surjective. Moreover ker$(\phi)=\{ae_1+be_2\mid a,b\in k^\times\}\cong k^\times\times k^\times$.

Set $\op{Sp}(M_1,M_2)=\op{GL}(M,B)\cap\op{GL}(M_1,M_2)$. If\/ $\tau$ is transposition then
$$
\op{Sp}(M_1,M_2)=\{(g,sg^{-1}s)\in\op{GL}(M_1)\times\op{GL}(M_2)\mid sgs=g^{-\tau}\}:\langle s\rangle
$$
Now let $\op{PGL}(M_1,M_2,\sigma)$ be the centralizer of\/ $\sigma$ in $\op{PGL}(M_1,M_2)$. Then $\phi$ restricts to a surjective map $\theta:\op{Sp}(M_1,M_2)\rightarrow\op{PGL}(M_1,M_2,\sigma)$. The kernel of\/ $\theta$ is $K=\{(a,a^{-1})\mid a\in k^\times\}$. So $K\cong k^\times$ and $s$ inverts each element of\/ $K$.

\begin{Lemma}\label{L:E1xE2}
Each projective representation $\rho:G\rightarrow\op{PGL}(M_1,M_2,\sigma)$ is realised by a group representation $\chi:H\rightarrow\op{Sp}(M_1,M_2)$ which arises from a commutative diagram of finite groups with exact rows
$$
\begin{CD}
1@>>>O@>{\rm inc}>>H@>\nu>>G@>>>1\\
 && @VV{\eta}V@VV{\chi}V@VV{\rho}V&&\\
1@>>>K@>{\rm inc}>>\op{Sp}(M_1,M_2)@>{\theta}>>\op{PGL}(M_1,M_2,\sigma)@>>>1\\
\end  {CD}
$$
Here $O$ is a cyclic group of odd order, $\eta$ is injective, ${[H\!:\!C_H(O)]\!\leq\! 2}$ and all elements of\/ $H\backslash C_H(O)$ invert $O$.
\begin{proof}
The pull-back diagram associated with $\rho$ and $\theta$ is 
$$
\begin{CD}
1@>>>K@>{\rm inc}>>\hat{G}@>\nu>>G@>>>1\\
 && @VV{=}V@VV{\chi}V@VV{\rho}V&&\\
1@>>>K@>{\rm inc}>>\op{Sp}(M_1,M_2)@>{\theta}>>\op{PGL}(M_1,M_2,\sigma)@>>>1\\
\end  {CD}
$$

Every element of\/ $G$ centralizes or inverts $K$. In this way $K\cong k^\times$ is a (possibly non-trivial) ${\mathbb Z}G$-module. Set $\gamma(\lambda)=\lambda^{|G|}$, for all $\lambda\in K$. As $k$ is algebraically closed, $\gamma$ is a surjective endomorphism of\/ $K$. We have a short exact sequence of abelian groups
$$
\begin{CD}
1@>>>O@>{\eta}>>K@>\gamma>>K@>>>1.
\end{CD}
$$
Here $O$ is the set of roots of\/ $x^{|G|}-1$ in $k$. So $O$ is a finite group. This induces a long exact sequence in cohomology, including
$$
\begin{CD}
\dots@>>>{\rm H}^2(G,O)@>{\eta_*}>>{\rm H}^2(G,K)@>\gamma_*>>{\rm H}^2(G,K)@>>>\dots.
\end{CD}
$$
Now $\gamma_*$ is the zero map, as multiplication by $|G|$ annihilates ${\rm H}^2(G,K)$. Let $d\in{\rm H}^2(G,K)$ be the factor set associated with
$$
\begin{CD}
1@>>>K@>{\rm inc}>>\hat{G}@>\nu>>G@>>>1.
\end{CD}
$$
Then there exists $c\in{\rm H}^2(G,O)$ mapping onto $d$. This gives us the commutative diagram in the statement of the Lemma.
\end{proof}
\end{Lemma}

We mention that Theorem A.5 in \cite{MurrayPFS} wrongly claims (in the notation used here) that $H$ is a central extension of\/ $G$. Now \cite[7.2]{MurrayPFS} relies on Theorem A.5, but does not require $O\leq \op{Z}(H)$. So 7.2 is still correct.

\section{Induction and Bilinear Forms}\label{S:Induction}

Throughout this section $M$ is a $kG$-module, $B$ is a symmetric $G$-form on $M$ and $\sigma$ is the adjoint of $B$ on $E(M)$. In many results we could take $B$ to be symplectic. We will require $\op{char}(k)=2$ in part \ref{SS:2applications}.

\subsection{Induction and Mackey's Formula}\label{SS:Induction}

For $H$ a subgroup of $G$ the trace map ${\op{tr}}_H^G:E_H(M)\rightarrow E_G(M)$ is the $k$-linear map ${\op{tr}}_H^G(f):=\sum{^g\!}f$, for ${f\in E_H(M)}$, where $g$ ranges over a left transversal $G/H$ to $H$ in $G$. Its image is an ideal of\/ $E_G(M)$, denoted $E_H^G(M)$. The $\sigma$-invariant elements in $E_H(M)$ form a subspace, but not an ideal, of\/ $E_H(M)$.

If $L$ is a $kH$-module, the induced module $\op{Ind}_H^GL$ is a direct sum of the $k$-vector spaces ${^g\!}L$, as $g$ ranges over $G/H$. Here ${^g\!}L=g\otimes L$ is the $k{^g\!}H$-module such that $ghg^{-1}(g\otimes \ell)=g\otimes h\ell$, for all $h\in H$ and $\ell\in L$. Recall that $M$ is said to be $H$-projective if $M\mid\op{Ind}_H^GL$, for some $kH$-module $L$.

Let $B_L$ be a symmetric $H$-form on $L$. The induced symmetric $kG$-module ${\op{Ind}}_H^G(L,B_L)$ is ${\op{Ind}}_H^GL$ with the induced $G$-form $B_L^G$, where
$$
B_L^G(g_1\otimes \ell_1,g_2\otimes \ell_2)=
\left\{
\begin{array}{ll}
B_L(g_2^{-1}g_1\ell_1,\ell_2),&\quad\mbox{if\/ $g_1H=g_2H$.}\\
0,&\quad\mbox{if\/ $g_1H\ne g_2H$.}
\end{array}
\right.
$$
Let ${^g\!}B_L$ denote the restriction of\/ $B_L^G$ to ${^g\!}L$. Then ${\op{Ind}}_H^G(L,B_L)$ is the orthogonal direct sum of the symmetric $k$-spaces $({^g\!}L,{^g\!}B_L)$. It is clear that $B_L^G$ is symplectic if and only if\/ $B_L$ is symplectic.

There is a version of Mackey's formula \cite[3.1.9]{NagaoTsushima} for symmetric modules. However its usefulness is limited by the absence of the Krull-Schmidt theorem:

\begin{Lemma}\label{L:Mackey}
Given ${K\leq G}$, there is an isomorphism of symmetric $kK$-modules 
$$
\op{Res}_K^G\op{Ind}_H^G(L,B_L)\cong\mathop{\perp}\limits_{g\in K\backslash G/H}\op{Ind}_{K\cap{^g\!}H}^K\op{Res}_{K\cap{^g\!}H}^{{^g\!}H}({^g\!}L,{^g\!}B_L).
$$
\begin{proof}
Let $g\in G$. The assignment $i\rightarrow ig$ maps $K/{K\cap{^g\!}H}$ to a set of representatives for the left cosets of\/ $H$ in $KgH$ and $\sum_{i\in K/K\cap{^g\!}H}{^{ig}\!}L\cong\op{Ind}_{K/K\cap{^g\!}H}^K{^g\!}L$. This induces an isomorphism $\sum_{i\in K/K\cap{^g\!}H}({^{ig}\!}L,{^{ig}\!}B_L)\cong\op{Ind}_{K/K\cap{^g\!}H}^K({^g\!}L,{^g\!}B_L)$ of symmetric $kK$-modules.
\end{proof}
\end{Lemma}

\subsection{Higman's Criterion}\label{SS:Higman}

Higman's Criterion \cite[4.2.2]{NagaoTsushima} is the assertion that the following are equivalent:
\begin{enumerate}
 \item[(1)] $M$ is $H$-projective.
 \item[(2)] $\op{tr}_H^G(\alpha)$ is a unit in $E_G(M)$, for some $\alpha\in E_H(M)$.
 \item[(3)] $M\mid\op{Ind}_H^G\op{Res}_H^G M$.
\end{enumerate}
We will generalise the equivalence $(1)\Longleftrightarrow(2)$ to symmetric $kG$-modules. However, the analogue of (3) is a strictly stronger statement.

Recall from Section \ref{SS:Endomorphisms} that $\{B_\theta\mid\theta\in E_G(M)^\times,\theta^\sigma=\theta\}$ give all non-degenerate $G$-invariant symmetric bilinear forms on $M$. Here are some important definitions:

\begin{Definition}\label{D:symmetricProjectivity}
Let $\theta$ be a $\sigma$-invariant unit in $E_G(M)$. Then we say that
\begin{itemize}
\item $\theta$ is $(H,\sigma)$-projective if\/ ${\theta=\op{tr}_H^G(\alpha)}$ for some $\alpha\in E_H(M)$ with $\alpha^\sigma=\alpha$.
\item $B_\theta$ is $H$-projective if\/ $(M,B_\theta)$ is a component of\/ $\op{Ind}_H^G(L,B_L)$ for some symmetric $kH$-module $(L,B_L)$.
\item $M$ is symmetrically $H$-projective if it affords a non-degenerate $G$-invariant symmetric bilinear form which is $H$-projective.
\end{itemize}
\end{Definition}
We note that if $B_\theta$ is $H$-projective, then it is projective relative to every subgroup of $G$ containing a $G$-conjugate of $H$.

\begin{Lemma}\label{L:Higman}
Let $\theta$ be a $\sigma$-invariant unit in $E_G(M)$. Then $\theta$ is $(H,\sigma)$-projective if and only if\/ $B_\theta$ is $H$-projective. More precisely, if\/ $\theta={\op{tr}}_H^G(\alpha)$ for some $\sigma$-invariant $\alpha\in E_H(M)$, then there is a $kG$-isometry $(M,B_\theta)\rightarrow{\op{Ind}}_H^G(\alpha M,\hat B_\alpha)$.
\begin{proof}
Here $\hat B_\alpha$ is the $H$-form on $\alpha M$ defined in Lemma \ref{L:PerfectGPairing}.

Suppose first that $\theta={\op{tr}}_H^G(\alpha)$ for some $\sigma$-invariant $\alpha\in E_H(M)$. Then $1\otimes\alpha$ can be regarded as a $kH$-homomorphism $M\rightarrow\op{Ind}_H^G(\alpha M)$. So $\phi={\op{tr}}_H^G(1\otimes\alpha)$ is a $kG$-homomorphism $M\rightarrow{\op{Ind}}_H^G(\alpha M)$. Now
$$
\begin{array}{lll}
\hat B_\alpha^G(\phi m_1,\phi m_2)
&=\sum\limits_{g\in G/H}\hat B_\alpha(\alpha g m_1,\alpha g m_2)
&=\sum\limits_{g\in G/H}B(\alpha g m_1,g m_2)\\
&=B({\op{tr}}_H^G(\alpha)m_1,m_2)
&=B_\theta(m_1,m_2).
\end{array}
$$
So $\phi$ defines an isometry $(M,B_\theta)\rightarrow{\op{Ind}}_H^G(\alpha M,\hat B_\alpha)$. Thus $B_\theta$ is $H$-projective.

Now suppose that $B_\theta$ is $H$-projective. Then there is a $kG$-isometry $\phi:(M,B_\theta)\rightarrow{\op{Ind}}_H^G(L,B_1)$ for some symmetric $kH$-module $(L,B_1)$. Let $e\in E_H(\op{Ind}_H^GL)$ be the orthogonal projection onto $1\otimes L$. Now $B_1(e\phi\_,e\phi\_)$ is a symmetric $H$-form on $M$. So there exists $\alpha\in E_H(M)$ such that $\alpha^\sigma=\alpha$ and
$$
B_1(e\phi m_1,e\phi m_2)=B(\alpha m_1,m_2),\quad\mbox{for all $m_1,m_2\in M$.}
$$
As $\phi$ is an isometry, we have
$$
\begin{array}{lll}
B_\theta(m_1,m_2)
&=B_1^G(\phi m_1,\phi m_2)
&=\sum\limits_{g\in G/H}B_1(e\phi g m_1,e\phi g m_2)\\
&=\sum\limits_{g\in G/H}B(\alpha g m_1,gm_2)
&=B({\op{tr}}_H^G(\alpha)m_1,m_2).
\end{array}
$$
So $\theta={\op{tr}}_H^G(\alpha)$ is $(H,\sigma)$-projective.
\end{proof}
\end{Lemma}

We now prove an analogue of Higman's Criterion for symmetric modules:

\begin{Proposition}\label{P:Higman}
The following are equivalent:
\begin{enumerate}
\item[(1)\hphantom{$'$}] $M$ is symmetrically $H$-projective.
\item[(2)\hphantom{$'$}] $\op{tr}_H^G(\alpha)$ is a unit in $E_G(M)$, for some $\sigma$-invariant $\alpha\in E_H(M)$.
\item[(2)$'$] $B_\theta$ is $H$-projective for some $\sigma$-invariant unit $\theta\in E_G(M)$.
\end{enumerate}
\begin{proof}
(1) and (2)$'$ are equivalent from the definitions. Lemma \ref{L:Higman} shows that (2) and (2)$'$ are equivalent.
\end{proof}
\end{Proposition}

Be aware that if $(M,B_\theta)$ is a component of $\op{Ind}_H^G\op{Res}_H^G(M,B_\theta)$ then it does not follow that it is a component of $\op{Ind}_K^G\op{Res}_K^G(M,B_\theta)$ when $H\leq K\leq G$. So we merely state the following analogue of condition (3) in Higman's criterion:

\begin{Lemma}\label{L:H-strong_proj}
Let $\theta$ be a $\sigma$-invariant unit in $E_G(M)$. Then
$$
(M,B_\theta)\mid\op{Ind}_H^G\op{Res}_H^G(M,B_\theta)
$$
if and only if ${{\op{tr}}_H^G(\alpha\theta\alpha^\sigma)=\theta}$, for some ${\alpha\in E_H(M)}$.
\end{Lemma}

\subsection{Green and Symmetric Vertices}\label{SS:2applications}

Now assume that $M$ is indecomposable and $\op{char}(k)=2$. We prove the first statement of Theorem \ref{T:vertex}. The main technical problem is that trace maps do not behave well with respect to adjoints. So Mackey's formula for the product of a pair of traces is not useful in our context. Instead we cancel pairs of terms in {\em triple} products of traces.

\begin{Lemma}\label{L:GreenSymmetric}
Each symmetric vertex of\/ $M$ contains a Green vertex of\/ $M$ with index at most $2$.
\begin{proof}
Let $T$ be a symmetric vertex of\/ $M$ and let $V$ be a Green vertex of $M$ which is contained in $T$. There is nothing to prove if\/ $V=T$. So assume that $V\ne T$.

By Higman's criterion $1_M=\op{tr}_V^G(\alpha)$, for some $\alpha\in E_V(M)$. Now let $\theta$ be a $\sigma$-invariant unit in $E_G(M)$ which is $(T,\sigma)$-projective. So $\theta={\op{tr}}_T^G(\beta)$ for some $\sigma$-invariant $\beta\in E_T(M)$. Consider the triple product
$$
\theta={\op{tr}}_V^G(\alpha){\op{tr}}_T^G(\beta){\op{tr}}_V^G(\alpha^\sigma)=\sum_{a,b\in G/V,c\in G/T}({^a}\alpha)({^c}\beta)({^b}\alpha^\sigma).
$$

Each $G$-orbit in $G/V\times G/T\times G/V$ contains a 3-tuple of the form $(aV,T,bV)$. We say that the orbit is:
\begin{itemize}
\item diagonal if\/ $aV=bV$,
\item symmetric if\/ $aV\ne bV$ but the orbit contains $(bV,T,aV)$,
\item antisymmetric if the orbit does not contain $(bV,T,aV)$.
\end{itemize}
We denote the collections of such orbits by ${\mathcal O}_d,{\mathcal O}_s$ and ${\mathcal O}_a$ respectively.

The stabilizer of\/ $(aV,T,bV)$ is ${^a}V\cap{^b}V\cap T$. So the orbit sum is
$$
\op{tr}(a,b):={\op{tr}}_{{^a}V\cap{^b}V\cap T}^G({^a}\alpha\,\beta({^b}\alpha^\sigma))\in E_G(M).
$$
Now $\op{tr}(a,b)^\sigma=\op{tr}(b,a)$. So $\theta$ is a sum, in $E_G(M)$, of\/ $\sigma$-invariant terms 
$$
\theta=\sum_{{\mathcal O}_d}\op{tr}(a,a)+\sum_{{\mathcal O}_s}\op{tr}(a,b)+\sum_{{\mathcal O}_a}(\op{tr}(a,b)+\op{tr}(b,a)).
$$

Write $\op{tr}(a,b)=\lambda1_M+j$, with $\lambda\in k^\times$ and $j\in J(E_G(M))$. Then for each pair of anti-symmetric orbits  $\op{tr}(a,b)+\op{tr}(b,a)=j+j^\sigma$ belongs to $J(E_G(M))$. Suppose that $\op{tr}(a,a)$ is a unit in $E_G(M)$, for some diagonal orbit. Then $B_{\op{tr}(a,a)}$ is a $({^a\!}V\cap T)$-projective symmetric $G$-form on $M$. This is impossible, as ${^a\!}V\cap T\lneq T$.

Now $\theta$ is a unit in the local ring $E_G(M)$. So we can choose a triple $(aV,T,bV)$ in a symmetric orbit such that $\op{tr}(a,b)$ is a unit. We then replace $V$ by a conjugate so that $a=1$, to simplify the notation. Then $\op{tr}(1,b)$ is a unit and $(V,T,bV)$ is in a symmetric $G$-orbit.

As $(bV,T,V)$ is $G$-conjugate to $(V,T,bV)$ there is $t\in T$ with ${tV=bV}$ and ${tbV=V}$. So $t\in N_T(V\cap{^b}V)$ and ${t^2\in V\cap{^b}V}$. Then $VbV=Vb^{-1}V$ is a self-dual double coset and $[(V\cap{^b}V)\langle t\rangle:V\cap{^b}V]=2$.

Set $\gamma:=\alpha\,\beta({^b}\alpha^\sigma)$. Then $\gamma^\sigma={^b}\alpha\,\beta \alpha^\sigma=\gamma^t$. So $\gamma+{^t}\gamma$ is fixed by both $\sigma$ and $(V\cap{^b}V)\langle t\rangle\cap T$. Moreover $\op{tr}(1,b)=\op{tr}_{(V\cap{^b}V)\langle t\rangle\cap T}^G(\gamma+{^t}\gamma)$. As $T$ is a symplectic vertex of\/ $M$, this forces $(V\cap{^b}V)\langle t\rangle=T$. But $V\ne T$ and $t^2\in V$. We deduce that $V={^b}V$, $T=V\langle t\rangle$ and $[T:V]=2$. 
\end{proof}
\end{Lemma}

We will use our second result in Proposition \ref{P:case1}.

\begin{Lemma}\label{L:Indecomposable}
Suppose that $M$ is symmetrically $H$-projective. Then $M$ is a non-degenerate component of some $\op{Ind}_H^G(L,B_L)$ where $L$ is indecomposable.
\begin{proof}
There is a $kG$-isometry $\phi:(M,B_\theta)\rightarrow\op{Ind}_H^G(L,B_L)$, where $(L,B_L)$ is a symmetric $kH$-module and $\theta$ is a $\sigma$-invariant unit in $E_G(M)$. Choose this $\phi$ with dim$(L)$ minimal. So $(L,B_L)$ is orthogonally indecomposable, by Lemma \ref{L:M->Sum}.

We claim that $L$ is indecomposable. For otherwise $L=L_1\dot+ L_2$, with $L_1^*\cong L_2$, according to Lemma \ref{L:GowWillems}. Let $e\in E_H(\op{Ind}_H^GL)$ be orthogonal projection onto $1\otimes L$ and let $e_i=ee_ie$ be projection onto $1\otimes L_i$ with kernel $1\otimes L_{3-i}$, for $i=1,2$. Now there are $\alpha_{ij}\in E_H(M)$, for $i,j=1,2$, such that 
$$
B_L(e_i\phi m_1,e_j\phi m_2)=B(\alpha_{ij}m_1,m_2),\quad\mbox{for all $m_1,m_2\in M$.}
$$
As $e=e_1+e_2$ and $\phi$ is an isometry, we get
$$
\sum_{i,j=1,2}{\op{tr}}_H^G(\alpha_{ij})=\theta.
$$
Now it can be shown that $\alpha_{ij}^\sigma=\alpha_{ji}$. So $\op{tr}_H^G(\alpha_{12})+\op{tr}_H^G(\alpha_{21})=\op{tr}_H^G(\alpha_{12})+\op{tr}_H^G(\alpha_{12})^\sigma$. This sum belongs to $J(E_G(M))$, as $E_G(M)$ is a local ring and $\op{char}(k)=2$. So ${\theta_{ii}\!:=\!{\op{tr}}_H^G(\alpha_{ii})}$ is a unit in $E_G(M)$ for some $i$. Lemma \ref{L:Higman} gives an isometry $(M,B_{\theta_{ii}})\rightarrow{\op{Ind}}_H^G(\alpha_{ii}M,\hat B_{\alpha_{ii}})$. But dim$(\alpha_{ii}M)<{\rm dim}(L)$, contradicting our choice of $L$. This establishes our claim and completes the proof.
\end{proof}
\end{Lemma}

\section{Symmetric vertices}\label{S:SymmetricVertices}

For the rest of the paper $k$ is an algebraically closed field of characteristic $2$.

\subsection{Projective modules and involutions}\label{SS:Projective}

We interpret some of the results of \cite{GowWillemsPIMs} and \cite{MurrayPFS} in terms of induction of forms. The structural map of the left regular module $kG$ is $\ell:kG\rightarrow E(kG)$, where $\ell(x)y:=xy$, for $x,y\in kG$. We also consider the $k$-algebra isomorphism $r:kG^{op}\rightarrow E_G(kG)$, defined by $r(x)y:=yx$. Also $kG$ is a left $k(G\times G)$-module via $(g_1,g_2)x:=g_1xg_2^{-1}$, for $g_1,g_2\in G$ and $x\in kG$.

Recall that $B_1$ is the $G\times G$-form on $kG$ such that for all $g_1,g_2\in G$ in $kG$
$$
B_1(g_1,g_2)=\left\{\begin{array}{ll}1_k,&\quad\mbox{if $g_1=g_2$,}\\0_k,&\quad\mbox{if $g_1\ne g_2$.}\end{array}\right.
$$
We will use $\sigma$ to denote the adjoint of $B_1$ on $E(kG)$. Then $\ell(x)^\sigma=\ell(x^o)$ and $r(x)^\sigma=r(x^o)$, for all $x\in kG$, where $^o$ is the contragredient map on $kG$.

Recall from \ref{SS:Endomorphisms} that the $G$-invariant bilinear forms on $kG$ are $B_{r(a)}$, for $a\in kG$. We simplify $B_{r(a)}$ to $B_a$. So ${B_a(x,y)=B_1(xa,y)}$, for all ${x,y\in kG}$. Then $B_a$ is non-degenerate if\/ $a$ is a unit in $kG$, symmetric if\/ $a=a^o$ and symplectic if\/ $a=a^o$ and $B_1(a,1)=0$. As an important example, $B_t$ is a non-degenerate symplectic form, for each involution $t\in G$.

Now suppose that $a$ is a symmetric unit in $kG$ cf. \cite{Lee}. So $a=a^o$. Set $\sigma_a$ as the adjoint of $B_a$ on $E(M)$. Then it is easy to check that
\begin{equation}\label{E:left_right}
\ell(x)^{\sigma_a}=\ell(x^o)\quad\mbox{and}\quad r(x)^{\sigma_a}=r(ax^oa^{-1}),\quad\mbox{for all $x\in kG$.}
\end{equation}
The fact that the adjoint has a different effect on the isomorphic subrings $\ell(kG)$ and $r(kG)$ of $E(kG)$ has been neglected in previous works such as \cite{LandrockManz} and \cite{GowWillemsPIMs}.

Let $P$ be a self-dual principal indecomposable $kG$-module. If $P$ is the projective cover of the trivial module then it affords a diagonalizable form and at least one symplectic form. Otherwise $P$ has symmetric type if and only if it has symplectic type. In fact, each non-degenerate $G$-invariant symplectic form on $P$ is the polarization of a $G$-invariant quadratic form. For full details, see \cite{GowWillemsPIMs}.

Our first result includes {\em Fong's Lemma}:

\begin{Lemma}\label{L:Fong}
Let $M$ be a non-trivial self-dual irreducible $kG$-module. Then $M$ affords a unique, up to scalars, non-degenerate $G$-invariant symplectic bilinear form.

The non-symplectic form $B_1$ is degenerate on each direct summand of\/ $kG$ which is isomorphic to $P(M)$.
\begin{proof}
There is a point $\epsilon$ of $E_G(M)$ such that $P(M)\cong kGe$ for $e\in\epsilon$. Then $\pi_\epsilon:kG\rightarrow E(M)$ is surjective. As $P(M)\cong P(M)^*$, we have $\op{ker}(\epsilon)^\sigma=\op{ker}(\epsilon)$. So $\sigma$ induces an involution, also denoted $\sigma$, on $E(M)$. Let $B$ be the non-degenerate symmetric form on $M$ whose adjoint is $\sigma$, as given by Lemma \ref{L:sigma=B}. Then $B$ is $G$-invariant as $B_1$ is $G$-invariant, symplectic as $M$ has no trivial submodule, and unique up to multiplication by a non-zero scalar, as $E_G(M)\cong k$.

We claim that $B_1$ is degenerate on $kGe$, for all $e\in\epsilon$. Otherwise Lemma \ref{L:projection} shows that $r(e)^\sigma=r(e)$ for some $e\in\epsilon$. But then $\ell(e)^\sigma=\ell(e)$. So $\pi_\epsilon(e)$ is an $\sigma$-invariant primitive idempotent in $E(M)$. Again by Lemma \ref{L:projection}, $B$ is non-degenerate on the 1-dimensional space $\pi_\epsilon(e)M$. This contradicts the fact that $B$ is symplectic and our claim follows.
\end{proof}
\end{Lemma}

Our next result includes a proof of \cite[(1.6)]{GowWillemsPIMs}.

\begin{Lemma}\label{L:P(k)}
$B_1$ restricts to a diagonalizable $G$-form on each direct summand of\/ $kG$ which is isomorphic to $P(k_G)$. Also $\dim P(k_G)/|G|_2$ is odd.
\begin{proof}
We may write $kG\cong\sum P(S)^{\dim S}$ where $S$ ranges over all irreducible $kG$-modules. As $P(S)\cong P(k_G)^*$ if and only if $S\cong k_G$, Lemma \ref{L:GowWillems} implies that $B_1$ is non-degenerate on each direct summand of\/ $kG$ which is isomorphic to $P(k_G)$. Now $B_1$ is symplectic on each direct summand of $kG$ which is not isomorphic to $P(k_G)$. As $B_1$ is not symplectic on $kG$, it is not symplectic on any direct summand isomorphic to $P(k_G)$.

Now $|G|_{2'}=\sum({\dim P(S)}/{|G|_2})\dim S$, where each ${\dim P(S)}/{|G|_2}$ is an integer. If\/ $S\cong S^*$ and $S\not\cong k_G$ then $\dim S$ is even, by Fong's Lemma. If\/ $S\not\cong S^*$, then $S$ and $S^*$ contribute equally to the sum. Thus $|G|_{2'}\equiv{\dim P(k_G)}/{|G|_2}$ (mod $2$).
\end{proof}
\end{Lemma}

Recall that $E(kG)$ is a $G$-algebra, as $kG$ is a $kG$-module. So by definition ${^g}f(z):=gfg^{-1}(z)$, for all $g\in G,f\in E(kG)$ and $z\in kG$. To describe the relative trace maps $\op{tr}_H^G:E_H(kG)\rightarrow E_G(kG)$, we first need a description of $E(kG)$. We use $B_1$ to identify $kG\otimes kG$ with $E(kG)$: $x\otimes y$ is the rank-1 endomorphism
\begin{equation}\label{E:E(kG)}
(x\otimes y)(z)=B_1(y,z)x,\quad\mbox{for $z\in kG$.}
\end{equation}
Then the structure of the involutary $G$-algebra $(E(kG),\sigma)$ is given by
\begin{equation}\label{E:sigma}
 (x\otimes y)^\sigma=y\otimes x,\quad\mbox{and}\quad{^g\!}(x\otimes y)=gx\otimes gy,\quad\mbox{for $g\in G$.}
\end{equation}
Using this we see that
$$
{\op{tr}}_1^G(x\otimes y)=r(y^ox),\quad\mbox{for all $x,y\in kG$.}
$$

It is useful to list the elements of\/ $G$ as
$$
1,\,t_1,\dots,t_m,
\begin{array}{llll}
    &g_1,&\dots,&g_n\\
    &g_1^{-1},&\dots,&g_n^{-1}
\end{array}
$$
where each $t_i$ is an involution.

\begin{Lemma}
A basis for the $\sigma$-invariant elements in $E_G(kG)$ is
$$
r(1),r(t_1),\dots r(t_m),r(g_1+g_1^{-1}),\dots,r(g_n+g_n^{-1}).
$$
Of these $r(1)$ and $r(g_j+g_j^{-1})$ are $(1,\sigma)$-projective, while $r(t_i)$ is $(H,\sigma)$-projective, for $H\leq G$, if and only if\/ ${^g}t_i\in H$, for some $g\in G$.
\begin{proof}
The first statement follows from $kG^{op}\cong E_G(kG)$ and \eqref{E:left_right}. Let $g\in G$. Then $1\otimes1$ and $g\otimes1+1\otimes g$ are $\sigma$-invariant and
$$
{\op{tr}}_1^G(1\otimes1)=r(1),\quad{\op{tr}}_1^G(g\otimes1+1\otimes g)=r(g+g^{-1}).
$$
So $r(1)$ and $r(g+g^{-1})$, for $g\ne g^{-1}$, are $(1,\sigma)$-projective.

Let $t=t_i$. Then $t\otimes1+1\otimes t\in E_{\langle t\rangle}(M)$ is $\sigma$-invariant and
$$
{\op{tr}}_{\langle t\rangle}^G(t\otimes1+1\otimes t)=r(t).
$$
So $r(t)$ is $(\langle t\rangle,\sigma)$-projective, if\/ $t\ne1$.

Let $H$ be a subgroup of\/ $G$. Then the endomorphisms ${\op{tr}}_1^H(a\otimes b)$ span $E_H(M)$, as $a,b$ range over all elements of\/ $G$. Now
$$
{\op{tr}}_1^H(a\otimes b)(g)=
\left\{
\begin{array}{ll}
gb^{-1}a,&\quad\mbox{if\/ $g\in Hb$.}\\
0,&\quad\mbox{if\/ $g\in G\backslash Hb$.}
\end{array}
\right.
$$
So ${\op{tr}}_1^H(a\otimes b)$ is $\sigma$-invariant if and only if\/ $Ha=Hb$ and $b^{-1}a=a^{-1}b$ i.e. if and only if\/ $t:=b^{-1}a$ is an involution such that ${^b}t\in H$. The last statement of the lemma follows from this.
\end{proof}
\end{Lemma}

\begin{Lemma}
Let $H\leq G$ and let $a$ be a unit in $kH$. Then
$$
{\op{Ind}_H^G(kH,B_a)\cong(kG,B_a)}.
$$
In particular $(kG,B_t)\cong\op{Ind}_{\langle t\rangle}^G(k\langle t\rangle,B_t)$, for all $t\in G$ with $t^2=1$.
\begin{proof}
Let $r_H(a)$ be the endomorphism $x\rightarrow xa$ of\/ $kH$. Then $r_H(a)$ extends to a $kH$-endomorphism of\/ $kG$ (acting as $0$ on $k(G\backslash H)$) and ${\op{tr}}_H^G(r_H(a))=r(a)$. The Lemma is a consequence of this fact. 
\end{proof}
\end{Lemma}

\begin{Lemma}\label{L:conjugate_involutions}
Two involutions $s,t\in G$ are $G$-conjugate if and only if\/ $(k\langle s\rangle,B_s)$ is a component of\/ $\op{Res}_{\langle s\rangle}^G(kG,B_t)$.
\begin{proof}
It is clear that there is an $\langle s\rangle$-isometry ${(k\langle s\rangle\!,B_s)\!\rightarrow\!\op{Res}_{\langle s\rangle}^G(kG\!,B_t)}$ if and only if\/ ${B_t(x,sx)\ne0}$, for some $x\in kG$. If\/ ${x=\sum_{g\in G}x_gg}$, with $x_g\in k$, then
$$
B_t(x,sx)=B_1(xt,sx)=\sum_{g\in G}x_gx_{sgt}=\sum_{g\in G,g=sgt}x_g^2,
$$
using $x_{g}x_{sgt}+x_{sgt}x_{g}=0$. So if\/ $B_t(x,sx)\ne0$ then $g=sgt$, for some $g\in G$. In that case $s=gtg^{-1}$ is conjugate to $t$.

Conversely, if\/ $s=gtg^{-1}$, then $B_t(g,sg)=1$.
\end{proof}
\end{Lemma}

Our next result is an elaboration of parts of \cite[Section 3]{GowWillemsPIMs}.

\begin{Lemma}\label{L:BtkGe}
Let $e$ be a primitive idempotent in $kG$ and let $\hat B$ be a non-degenerate $G$-invariant symplectic form on $kGe$. Then there is an involution $t\in G$ such that $B_t$ is non-degenerate on $kGe$ and $(k\langle t\rangle,B_t)$ is a component of\/ $\op{Res}_{\langle t\rangle}^G(kGe,\hat B)$.
\begin{proof}
By Lemma \ref{L:restrictions} there is $a\in ekGe^\sigma$ so that $\hat B(xe,ye)=B_a(x,y)$, for all $x,y\in kG$. Then $a=a^\sigma$ and $B_1(a,1)=0$, as $\hat B$ is symplectic. Write $a=\sum_{i=1}^m\alpha_it_i+\sum_{j=1}^n\beta_j(g_j+g_j^{-1})$, with $\alpha_i,\beta_j\in k$. Now $E_G(kGe)$ is a local ring. So each $B_{\beta_j(g_j+g_j^{-1})}$ is degenerate on $kGe$. It follows that $B_{\alpha_it_i}$ is non-degenerate on $kGe$, for some $i$. Set $t=t_i$. Then $B_t$ is non-degenerate on $kGe$ and $B_a(e,te)=\alpha_i\ne0$. So ${ke+kte}$ is a $B_a$-direct summand of\/ $\op{Res}_{\langle t\rangle}^G(kGe)$ which is isomorphic to $k\langle t\rangle$. We conclude that $(k\langle t\rangle,b_t)$ is a component of\/ $\op{Res}_{\langle t\rangle}^G(kGe,\hat B)$.
\end{proof}
\end{Lemma}

We note that $B_t$ is non-degenerate on $kGe$ if and only if $kGe=kGf$ for some primitive idempotent $f\in kG$ with $tft=f^o$, using Lemma \ref{L:projection} and \eqref{E:left_right}. This is \cite[3.4 and 3.5]{GowWillemsPIMs}. Our last result strengthens \cite[6.5]{MurrayPFS}.

\begin{Corollary}\label{C:BtkGe}
Suppose that $B_t$ is non-degenerate on $kGe$ where $t\in G$ and $t^2=1$. Then $(kGe,B_t)$ is a component of\/ $\op{Ind}_{\langle t\rangle}^G\op{Res}_{\langle t\rangle}^G(kGe,B_t)$.

Let $M$ be the irreducible head of\/ $kGe$. Then $\op{Res}_{C_G(t)}^GM$ has diagonalizable type. In particular $k_{C_G(t)}$ is a submodule and a quotient module of\/ $\op{Res}_{C_G(t)}^GM$.
\begin{proof}
If\/ $t=1$, then ${kGe\cong P(k_G)}$, and Lemma \ref{L:P(k)} gives all conclusions. So assume that $t$ is an involution. By Lemma \ref{L:BtkGe} there is an involution $s\in G$ and an isometry $(k\langle s\rangle,B_s)\rightarrow\op{Res}_{\langle s\rangle}^G(kGe,B_t)$. Then $s$ and $t$ are $G$-conjugate, according to Lemma \ref{L:conjugate_involutions}. This proves the first assertion.

We may assume that $M\not\cong k_G$. Let $B$ be the $G$-invariant symplectic form on $M$. Recall that $\pi_\epsilon:kG\rightarrow E(M)$ has kernel ${\mathfrak M}_\epsilon$. Now $\sigma_t$ is an involution on $kG$ which fixes ${\mathfrak M}_\epsilon$. So it induces an involution $\sigma_t$ on $E(M)$. As $\ell(g)^{\sigma_t}=\ell(tg^{-1}t)$, we get $\pi_\epsilon(g)^{\sigma_t}=\pi_\epsilon(tg^{-1}t)$. So the corresponding non-degenerate symmetric bilinear form on $M$ is $\hat B(m_1,m_2):=B(tm_1,m_2)$, for all $m_1,m_2\in M$. Then $\hat B$ is $C_G(t)$-invariant.

Let $f\in kG$ be the $\sigma_t$-invariant primitive idempotent with $kGe=kGf$, as given by Lemma \ref{L:projection}. Then $\pi_\epsilon(f)$ is a $\sigma_t$-invariant primitive idempotent in $E(M)$. It follows that $\hat B$ is a diagonalizable form on $M$. The last assertion now follows from Lemma \ref{L:diagonal=trivial}.
\end{proof}
\end{Corollary}

Finally we prove Theorem \ref{T:quadraticPIM} of the introduction. This gives an intriguing new criterion for a principal indecomposable module to be of quadratic type, which depends entirely on the irreducible head of the module.

\begin{Proposition}\label{P:quadraticPIM}
Let $M$ be a non-trivial self-dual irreducible $kG$-module, with symplectic form $B$, and let $t\in G$ be an involution. Then $P(M)$ is a $B_t$-component of\/ $kG$ if and only if\/ $B(tm,m)\ne0$, for some $m\in M$.
\begin{proof}
We proved the `only if' in Corollary \ref{C:BtkGe}. So suppose that there is $m\in M$ such that $B(tm,m)\ne0$. Define $\hat B$ on $M$ as above, with an involution $\sigma_t$ on $E(M)$ which lifts to the involution $\sigma_t$ of $kG$. Now $\hat B$ is non-degenerate on the subspace $km$ of $M$. So by Lemma \ref{L:projection} there is a unique $\sigma_t$-invariant primitive idempotent $\hat f\in E(M)$ such that $\hat f(M)=km$. By Lemma \ref{L:idempotent_lifting}(iii) there is a $\sigma_t$-invariant primitive idempotent $f\in kG$ such that $\pi_\epsilon(f)=\hat f$. In particular $kGf\cong P(M)$. Lemma \ref{L:projection} guarantees that $B_t$ is non-degenerate on $kGf$.
\end{proof}
\end{Proposition}

\begin{Example}\cite[11.2]{James}
Let $S_n$ be the symmetric group of degree $n\geq2$ and let $\lambda$ be a non-trivial $2$-regular partition of $n$. The Specht module $S^\lambda$ affords a non-zero $S_n$-invariant symplectic bilinear form $B$ such that $D^\lambda:=S^\lambda/(S^\lambda)^\perp$ is an irreducible $kS_n$-module. Let $T$ be a $\lambda$-tableau, so there is a corresponding polytabloid $e_T$ in $S^\lambda$. Now $T$ is a filling of the boxes in a Young diagram for $\lambda$ with the symbols $1,2,\dots,n$. Let $t$ be the involution in $S_n$ which reverses the entries in each row. Then the tabloid labelled by $T$ is the only tabloid common to $te_T$ and $e_T$. So $B(te_T,e_T)=1_k$. It follows that $P(D^\lambda)$ is a $B_t$-component of $kS_n$. Note that if $\lambda$ has $\ell$ non-zero parts then $t$ is a product of $(n-\ell)/2$ commuting transpositions.
\end{Example}

\subsection{Puig correspondence}\label{SS:Indecomposable}

In parts \ref{SS:Indecomposable} through to \ref{SS:Proofs} we take $M$ to be an indecomposable $kG$-module, with Green vertex $V$ and $V$-source $Z$.

We will use the following notation:
\begin{itemize}
 \item $\Omega$ is the point of\/ $E_G(\op{Ind}_V^GZ)$ such that $\omega\op{Ind}_V^GZ\cong M$, for all $\omega\in\Omega$.
 \item ${\mathcal F}:E(M)\rightarrow E(\op{Ind}_V^GZ)$ is the embedding of\/ $G$-algebras induced by $\Omega$.
 \item $\delta$ is the point of\/ $E_V(M)$ such that $dM\cong Z$, for all $d\in\delta$.
 \item $\Delta$ is the point of\/ $E_V(\op{Ind}_V^GZ)$ containing ${\mathcal F}(\delta)$.
\end{itemize}

Let $N_G(V,Z)$ be the stabilizer of\/ $Z$ in $N_G(V)$. Set ${N:=N_G(V,Z)/V}$. Then $E(P_\Delta)$ is an $N$-algebra. So $P_\Delta$ is a module for a twisted group algebra $k_\gamma N$ of\/ $N$ over $k$. Likewise $E(P_\delta)$ is an $N$-algebra and $P_\delta$ is a module for a twisted group algebra $k_{\gamma'}N$ of\/ $N$ over $k$. According to \cite[(26.1)]{Thevenaz}, $P_\Delta$ is the regular $k_\gamma N$-module. Now ${\mathcal F}$ induces an embedding of\/ $N$-algebras $E(P_\delta)\rightarrow E(P_\Delta)$. This in turn induces a group isomorphism between the central extensions of $N$ corresponding to the cocycles $\gamma$ and $\gamma'$. In this way, $P_\delta$ can and will be identified with an indecomposable component of $P_\Delta$, as $k_{\gamma}N$-modules.

\begin{Lemma}\label{L:kgamma}
Suppose that $Z\cong Z^*$ and either $Z$ or $M$ is of symmetric type. Then ${k_\gamma N\cong kN}$. So $P_\Delta$ is the regular $kN$-module.
\begin{proof}
Suppose first that $Z$ has a symmetric $V$-form $B_0$. Let $\sigma_0$ be the adjoint of\/ $B_0^G$ on $E(\op{Ind}_V^GZ)$. As $Z\cong Z^*$, Lemma \ref{L:PerfectGPairing} implies that $\Delta^{\sigma_0}=\Delta$. So ${\mathfrak M}_\Delta^{\sigma_0}={\mathfrak M}_\Delta$ and $\sigma_0$ is an involution on $E_V(\op{Ind}_V^GZ)/{\mathfrak M}_\Delta$. In this way $(E(P_\Delta),\sigma_0)$ is a simple involutary $N$-algebra.

By Lemma \ref{L:PGL(M,sigma)}, there is a symmetric form $B_{\sigma_0}$ on $P_\Delta$ such that the action of\/ $N$ on $E(P_\Delta)$ lifts to a representation $N\rightarrow\op{GL}(P_\Delta,B_{\sigma_0})$. In particular $k_\gamma N\cong kN$ as twisted group algebras and $P_\Delta$ is the regular $kN$-module. 

Conversely suppose that $M$ has a symmetric $G$-form $B$. Let $\sigma$ be the adjoint of\/ $B$ on $E(M)$. Then $\delta^\sigma=\delta$. So ${\mathfrak M}_\delta^\sigma={\mathfrak M}_\delta$ and $\sigma$ induces an involution on $E_V(M)/{\mathfrak M}_\delta$. According to Lemma \ref{L:PGL(M,sigma)}, the action of\/ $N$ on $E(P_\delta)$ lifts to a representation $N\rightarrow\op{GL}(P_\delta,B_{\sigma})$, where $B_\sigma$ is a symmetric form on $P_\delta$. Thus $k_{\gamma'}N\cong kN$ as twisted group algebras. But $k_\gamma N\cong k_{\gamma'}N$. So as before $P_\Delta$ is the regular $kN$-module.
\end{proof}
\end{Lemma}

Set ${N^*=N_G^*(V,Z)/V}$, where $N_G^*(V,Z)$ is the stabilizer of\/ $\{Z,Z^*\}$ in $N_G(V)$. So $[N^*:N]\leq2$. The following is based on \cite[(14.8)]{Thevenaz}.

\begin{Lemma}\label{L:onto}
Let $L$ be a component of\/ $\op{Ind}_V^GZ$ and let $\epsilon$ be a point of\/ $E_V(L)$ contained in $\Delta$. For ${V\leq H\leq G}$ set ${N_H\!=\!N_H(V,Z)/V}$ and ${N_H^*\!=\!N_H^*(V,Z)/V}$. Then for all ${f\!\in\! E_V(L)\epsilon E_V(L)}$ we have:
\begin{itemize}
\item[(i)] $\pi_\epsilon{\op{tr}}_V^H(f)={\op{tr}}_1^{N_H}\pi_\epsilon(f)$ and $\pi_\epsilon\op{res}_V^H:E_H(L)\rightarrow E_{N_H}(P_\epsilon)$ is onto.
\item[(ii)] If\/ $\sigma$ is a $G$-involution of\/ $E(L)$, then $\pi_{\epsilon,\epsilon^\sigma}{\op{tr}}_V^H(f)={\op{tr}}_1^{N_H^*}\pi_{\epsilon,\epsilon^\sigma}(f)$ and $\pi_{\epsilon,\epsilon^\sigma}\op{res}_V^H: E_H(L)\rightarrow(E(P_\epsilon)\times E(P_{\epsilon^\sigma}))^{N_H^*}$ is onto.
\end{itemize}
\begin{proof}
(i) follows from Remark (19.9) in \cite{Thevenaz} as ${E_H(L)=\op{tr}_V^H(E_V(L))}$.

From the proof of \cite[(14.7)]{Thevenaz} we see that $\pi_{\epsilon,\epsilon^\sigma}{\op{tr}}_V^H(f)={\op{tr}}_1^{N_H^*}(\pi_{\epsilon,\epsilon^\sigma}(f))$. So $\pi_{\epsilon,\epsilon^\sigma}$ maps $\op{tr}_V^H(E_V(L))$ onto ${\op{tr}}_1^{N_H^*}(E(P_\epsilon)\times E(P_{\epsilon^\sigma}))$. This is a $2$-sided ideal of the fixed point subalgebra $(E(P_\epsilon)\times E(P_{\epsilon^\sigma}))^{N_H^*}$ which is contained in $\pi_{\epsilon,\epsilon^\sigma} E_H(L)$.

We now modify \cite[(14.8)]{Thevenaz}. As $L\mid \op{Ind}_V^GZ$, we have $1_L={\op{tr}}_V^H(\iota)$ for some $\iota\in E_V(L)\epsilon E_V(L)$. So $1_{P_\epsilon+P_{\epsilon^\sigma}}={\op{tr}}_1^{N_H^*}\pi_{\epsilon,\epsilon^\sigma}(\iota)$ and hence
$$
{\op{tr}}_1^{N_H^*}(E(P_\epsilon)\times E(P_{\epsilon^\sigma}))=(E(P_\epsilon)\times E(P_{\epsilon^\sigma}))^{N_H^*}.
$$
Now (ii) follows from this and the previous paragraph.
\end{proof}
\end{Lemma}

In our situation the Puig correspondence \cite[(19.1)]{Thevenaz} is a multiplicity preserving bijection between the indecomposable components of\/ $\op{Ind}_V^GZ$ with vertex $V$ and the indecomposable components of\/ $P_\Delta$. More concretely, if\/ $e$ is a primitive idempotent in $E_G(\op{Ind}_V^GZ)$ such that $e\op{Ind}_V^GZ$ has vertex $V$, then $\pi_\Delta(e)$ is a primitive idempotent in $E_N(P_\Delta)$, and $\pi_\Delta(e)P_\Delta$ is the Puig correspondent of\/ $e\op{Ind}_V^GZ$.

\subsection{Theorem \ref{T:vertex}(i)}\label{SS:Part1}

\begin{Proposition}\label{P:PuigP(k_N)}
Suppose that $Z$ affords a non-degenerate $V$-invariant symmetric bilinear form $B_0$. Then the Puig correspondent of\/ $P(k_N)$ is the unique indecomposable $B_0^G$-component of\/ $\op{Ind}_V^GZ$ that has vertex $V$.
\begin{proof}
Lemma \ref{L:kgamma} applies, and we adopt its notation. Let $d\in\Delta$ be the orthogonal projection $\op{Ind}_V^G(Z)\rightarrow1\otimes Z$. So $d^{\sigma_0}=d$ and $\op{tr}_V^G(d)=1_{\op{Ind}_V^GZ}$. Set $\ov{d}:=\pi_\Delta(d)$. Then $1_{P_\Delta}={\op{tr}}_1^N(\ov{d})$, using Lemma \ref{L:onto}(i). Moreover $\ov{d}^{\sigma_0}=\ov{d}$. So $(P_\Delta,B_{\sigma_0})\mid\op{Ind}_1^N(\ov{d}P_\Delta,\hat B_{\ov d})$, according to Lemma \ref{L:Higman}. But dim$(\ov{d}P_\Delta)=1$. So $(\ov{d}P_\Delta,\hat B_{\ov d})\cong(k_1,B_1)$ and thus Ind$_1^N(\ov{d}P_\Delta,\hat B_{\ov d})\cong(kN,B_1)$. So we can and do identify $(P_\Delta,B_{\sigma_0})$ with $(kN,B_1)$.

Write $\op{Ind}_V^G(Z)=L_1\dot+L_2\dots+\dots\dot+L_n$, where the $L_i$ are indecomposable $kG$-modules and $L_1$ is the Puig correspondent of\/ $P(k_N)$. Then $L_i\not\cong L_1^*$ for $i>1$. So $B_0^G$ is non-degenerate on $L_1$, by Lemma \ref{L:GowWillems}.

Now suppose that $B_0^G$ is non-degenerate on $L_i$, where $L_i$ has vertex $V$. Then $L_i$ has $V$-source $Z$. Let $\omega\in E_G(\op{Ind}_V^GZ)$ be orthogonal projection onto $L_i$. Then $\pi_\Delta(\omega)$ is a $\sigma_0$-invariant primitive idempotent in $E_N(P_\Delta)$. So $\pi_\Delta(\omega)kN\cong P(k_N)$, by Lemma \ref{L:Fong}. We deduce that $L_i=L_1$.
\end{proof}
\end{Proposition}

From now on we assume that $M$ is of symmetric type.

\begin{Proposition}\label{P:case1}
The following are equivalent:
\begin{itemize}
\item[(i)] Each symmetric vertex of\/ $M$ is a Green vertex of $M$.
\item[(ii)] $Z$ has symmetric type and $M$ is the Puig correspondent of\/ $P(k_N)$.
\item[(iii)] $Z$ has symmetric type and if\/ $B_0$ is a $V$-form on $Z$ then $M$ is a $B_0^G$-component of\/ $\op{Ind}_V^GZ$.
\item[(iv)] $M$ has $G$-form $B$ such that $Z$ is a $B$-component of\/ $\op{Res}_V^GM$.
\end{itemize}
\begin{proof}
Assume (i). Then by Lemma \ref{L:Indecomposable}, there is symmetric $kV$-module $(Y,B_1)$ such that $Y$ is indecomposable and $M$ is a $B_1^G$-component of\/ $\op{Ind}_V^G(Y)$. Then $Y$ is a $V$-source of\/ $M$. But $Z={^n}Y$, for some $n\in N_G(V)$. So (iii) holds. Moreover, (i) and (ii) are equivalent, by Proposition \ref{P:PuigP(k_N)}.

Lemma \ref{L:kgamma} applies if (ii), (iii) or (iv) hold. We adopt its notation.

Assume (iii). Then $M$ has a $V$-projective $G$-form. So (i) is true. Lemma \ref{L:projection} implies that there is a $\sigma_0$-invariant idempotent $\omega$ in $\Omega$ and $B_0^G$ is non-degenerate on $\omega\op{Ind}_V^GZ\cong M$. Set $\ov{\omega}=\pi_\Delta(\omega)$. Then $\ov{\omega}kN\cong P(k_N)$. Lemma \ref{L:P(k)} implies that $B_1$ is non-degenerate on a $1$-dimensional subspace of\/ $\ov{\omega}kN$. So again using Lemma \ref{L:projection}, there is a $\sigma_0$-invariant primitive idempotent $\ov{d}\in\ov{\omega}E(kN)\ov{\omega}$, .

Now $\pi_\Delta$ restricts to a surjective map $\omega E_V(\op{Ind}_V^GZ)\omega\rightarrow\ov{\omega}E(kN)\ov{\omega}$. So by Lemma \ref{L:idempotent_lifting}(iii) there is a $\sigma_0$-invariant primitive idempotent $d\in \omega E_V(\op{Ind}_V^GZ)\omega$ such that $\pi_\Delta(d)=\ov{d}$. In particular $d\in\Delta$ and $d=\omega d\omega$. Then $B_0^G$ is non-degenerate on the $V$-component $d\op{Ind}_V^GZ$ of\/ $\omega\op{Ind}_V^G(Z)\cong M$. So (iv) holds.

Assume (iv). By Lemma \ref{L:projection} this means that there is $d\in\delta$ with $d^\sigma=d$. As $\pi_\delta(d)$ is a $\sigma$-invariant primitive idempotent in $E(P_\delta)$, $B_\sigma$ is a diagonalizable $N$-form on $P_\delta$. So $P_\delta\cong P(k_N)$, using Lemma \ref{L:Fong}. Now $P(k_N)$ is the only self-dual principal indecomposable $kN$-module which has multiplicity $1$ in $kN$. So $P_\delta\cong P(k_N)$, when regarded as a component of\/ $P_\Delta$. So (ii) holds.
\end{proof}
\end{Proposition}

The principal $2$-block ${\mathcal B}_0={\mathcal B}_0(G)$ of $G$ is the block of $kG$ which contains $k_G$.

\begin{Corollary}\label{C:principal}
If some Green vertex of $M$ is a symmetric vertex of\/ $M$ then $M$ belongs to the principal $2$-block of $G$.
\begin{proof}
From Proposition \ref{P:case1}, $M$ is the Puig correspondent of\/ $P(k_N)$. Set $C=VC_G(V)$. Then $\op{Res}_{C/V}^N(P(k_N))\cong P(k_{C/V})^m$, for some $m\geq1$. So $(V,P(k_{C/V}))$ is a root of\/ $M$, in the terminology of \cite{KulshammerRoots}. Now $P(k_{C/V})$ is in ${\mathcal B}_0(C)$, and Brauer's Third Main Theorem implies that $(V,{\mathcal B}_0(C))$ is a ${\mathcal B}_0$-subpair. So $M$ belongs to ${\mathcal B}_0$, according to \cite{KulshammerRoots}.
\end{proof}
\end{Corollary}

\subsection{Theorem \ref{T:vertex}(ii)}\label{SS:Part2}

\begin{Proposition}\label{P:case2}
Suppose that the sources of $M$ are self-dual, but $B$ is degenerate on each direct summand of $\op{Res}_V^GM$ that is a source of $M$. Then
\begin{itemize} 
\item[(i)] There is $V\leq T\leq N_G(V,Z)$ with $[T:V]=2$ such that $\op{Ind}_V^TZ$ is a $B$-component of\/ $\op{Res}_T^GM$.
\item[(ii)] Let $B_0$ be a $T$-form on $\op{Ind}_V^TZ$ such that $Z$ is not a $B_0$-component of\/ $\op{Res}_V^T\op{Ind}_V^TZ$. Then $M$ is a $B_0^G$-component of\/ $\op{Ind}_V^GZ$.
\item[(iii)] Either $V$ or $T$ is a symmetric vertex of\/ $M$.
\end{itemize}
\begin{proof}
Lemma \ref{L:kgamma} applies, and we adopt its notation.

As $B$ is degenerate on each direct summand of\/ $\op{Res}_V^GM$ that is isomorphic to $Z$, $\sigma$ does not fix any idempotent in $\delta$. So by Lemma \ref{L:idempotent_lifting}(i), $\sigma$ does not fix any primitive idempotent in $E(P_\delta)$. This means that $(P_\delta,B_\sigma)$ is a symplectic $kN$-module. Lemma \ref{L:BtkGe} gives an involution $t\in N$ such that $B_t$ is non-degenerate on $P_\delta$. Moreover $\op{Res}_{\langle t\rangle}^G(P_\delta,B_\sigma)$ has a component $(k\langle t\rangle,B_t)$. So there is a $\sigma$-invariant primitive idempotent $\ov{y}\in E_{\langle t\rangle}(P_\delta)$.

Let $T\leq N$ with $T/V=\langle t\rangle$. Then $\pi_\delta\op{res}_V^T:E_T(M)\rightarrow E_{\langle t\rangle}(P_\delta)$ is surjective, by Lemma \ref{L:onto}(i). So by Lemma \ref{L:idempotent_lifting}(iii), there is a primitive $\sigma$-invariant idempotent $y\in E_T(M)$ with $\pi_\delta(y)=\ov{y}$. So $Y=yM$ is a $B$-direct summand of\/ $\op{Res}_T^GM$. Now $Y$ is $V$-projective, $|T/V|=2$ and $Z$ is a component of\/ $Y_V$. So $Y\cong\op{Ind}_V^TZ$. The conclusion of (i) follows.

Assume the hypothesis of (ii) and set $Y:=\op{Ind}_V^TZ$. Note that $\op{Ind}_T^GY\cong\op{Ind}_V^GZ$. Let $\sigma_0$ be the adjoint of\/ $B_0^G$ on $E(\op{Ind}_V^GZ)$. Now $(E(P_\Delta),\sigma_0)$ is an involutary $N$-algebra. So $k_\gamma N\cong kN$, $P_\Delta$ is the regular $kN$-module and $\sigma_0$ is the adjoint of a symmetric $N$-form $B_{\sigma_0}$ on $P_\Delta$. By hypothesis on $B_0$, the $N$-form $B_{\sigma_0}$ is symplectic.

Let $e\in E_T(\op{Ind}_V^GZ)$ be orthogonal projection onto $1\otimes Y$. Then ${\op{tr}}_T^G(e)=1_{\op{Ind}_V^GZ}$ and ${e\in{\op{tr}}_V^T\left({E_V(\op{Ind}_V^GZ)}\,\Delta\,{E_V(\op{Ind}_V^GZ)}\right)}$. Set $\ov{e}=\pi_\Delta(e)$. Then
$$
1_{E(P_\Delta)}=\pi{\op{tr}}_T^G(e)={\op{tr}}_{\langle t\rangle}^{N}(\ov{e}),\quad\mbox{using Lemma \ref{L:onto}(i)}.
$$
As $\ov{e}^{\sigma_0}=\ov{e}$, Lemma \ref{L:Higman} gives a $kN$-isometry $(P_\Delta,B_{\sigma_0})\rightarrow\op{Ind}_{\langle t\rangle}^{N}(\ov{e}P_\Delta,\hat{B}_{\ov{e}})$. But dim$(\ov{e}P_\Delta)=2$. So this isometry is surjective, as both sides have dimension $|N|$. Now $B_{\sigma_0}$ is symplectic, and $\langle t\rangle$ is cyclic of order $2$. So $(\ov{e}P_\Delta,\hat{B}_{\ov{e}})\cong(k\langle t\rangle,B_t)$. We deduce that $(P_\Delta,B_{\sigma_0})\cong(kN,B_t)$.

Now $P_\delta$ is a $B_t$-component of\/ $P_\Delta$. So there is a primitive $\sigma_0$-invariant idempotent $\ov{\omega}\in E_N(P_\Delta)$ with $\ov{\omega}P_\Delta\cong P_\delta$. Since $\pi_\Delta{\op{res}}_V^G:E_G(\op{Ind}_V^GZ)\rightarrow E_N(P_\Delta)$ is surjective, Lemma \ref{L:idempotent_lifting}(iii) gives a primitive $\sigma_0$-invariant idempotent $\omega\in E_G(\op{Ind}_V^GZ)$ with $\pi_\Delta(\omega)=\ov{\omega}$. So $\omega\op{Ind}_V^GZ$ is a $B_0^G$-direct summand of\/ $\op{Ind}_V^GZ$. But $\omega\op{Ind}_V^G(Z)\cong M$. The conclusion of (ii) follows.

(iii) holds as $M$ has a $T$-projective symmetric $G$-form.
\end{proof}
\end{Proposition}

\subsection{Theorem \ref{T:vertex}(iii)}\label{SS:Part3}

\begin{Proposition}\label{P:case3}
Suppose that no source of $M$ is self-dual. Then
\begin{itemize}
\item[(i)] $N_G^*(V,Z)$ has a subgroup $T$ which contains $V$ such that $[T:V]=2$ and $\op{Ind}_V^TZ$ is a $B$-component of\/ $\op{Res}_T^GM$. Then $N^*={N\rtimes T/V}$.
\item[(ii)] $M$ is a $B_0^G$-component of\/ $\op{Ind}_V^GZ$, if $B_0$ is a symmetric $T$-form on $\op{Ind}_V^TZ$.
\item[(iii)] $T$ is a symmetric vertex of\/ $M$.
\end{itemize}
\begin{proof}
Set $\Delta^\sigma$ as the point of\/ $E_G(\op{Ind}_V^GZ)$ corresponding to $Z^*$. Lemma \ref{L:PerfectGPairing} implies that $\delta^\sigma\ne\delta$. So $({E(P_\delta)\!\times\!E(P_{\delta^\sigma}),\sigma})$ is an involutary $N^*$-algebra satisfying the hypothesis of Lemma \ref{L:sigma=B}(ii). Moreover, this algebra is embedded in the $N^*$-algebra ${E(P_\Delta)\!\times \!E(P_{\Delta^\sigma})}$ as follows. According to Lemma \ref{L:onto}(ii), the restriction map $\pi_{\Delta,\Delta^\sigma}:E_G(\op{Ind}_V^GZ)\!\rightarrow\!\left({E(P_\Delta)\!\times\!E(P_{\Delta^\sigma})}\right)^{N^*}$ is surjective. Let $\omega\in\Omega$. Then $\ov{\omega}:=\pi_{\Delta,\Delta^\sigma}(\omega)$ is a primitive idempotent in $(E(P_\Delta)\times E(P_{\Delta^\sigma}))^{N^*}$. We identify $E(P_\delta)\times E(P_{\delta^\sigma})$ with $\ov{\omega}\left(E(P_\Delta)\times E(P_{\Delta^\sigma})\right)\ov{\omega}$, and $P_\delta+P_{\delta^\sigma}$ with $\ov{\omega}(P_\Delta+P_{\Delta^\sigma})$.

By Lemma \ref{L:E1xE2} there is a commutative diagram
$$
\begin{CD}
1@>>>O@>{\rm inc}>>H@>\theta>>N^*@>>>1\\
 && @VV{\eta}V@VV{\chi}V@VV{\rho}V&&\\
1@>>>K@>{\rm inc}>>\op{Sp}(P_\delta,P_{\delta^\sigma})@>{\rho}>>\op{PGL}(P_\delta,P_{\delta^\sigma},\sigma)@>>>1\\
\end  {CD}
$$
where $O$ is a finite cyclic group of odd order and $\theta(C_H(O))=N$. Each element of\/ $H\backslash C_H(O)$ maps $P_\delta$ onto $P_{\delta^\sigma}$. Moreover $\sigma$ is the adjoint of a symplectic $H$-form $B_\sigma$ on $P_\delta+P_{\delta^\sigma}$, with
\begin{equation}\label{E:perps}
(P_\delta)^\perp=P_\delta\quad\mbox{and}\quad(P_{\delta^\sigma})^\perp=P_{\delta^\sigma}.
\end{equation}

Now $e_\eta=\frac{1}{|O|}\sum_{\lambda\in O}\eta(\lambda^{-1})\lambda$ is a central idempotent in $kH$ such that $P_\Delta+P_{\Delta^\sigma}\cong kHe_\eta$ as $kH$-modules. So $E_{N^*}(P_\Delta+P_{\Delta^\sigma})\cong e_\eta kHe_\eta$.

By Lemma \ref{L:BtkGe} there is an involution $t\in H$ such $B_t$ is non-degenerate on $P_\delta+P_{\delta^\sigma}$ and $k\langle t\rangle$ is a $B_\sigma$-component of\/ $\op{Res}^H_{\langle t\rangle}(P_\delta+P_{\delta^\sigma})$. This means that there is $p\in P_\delta+P_{\delta^\sigma}$ such that $B_\sigma(p,tp)\!\ne\!0$. Write $p=p_1+p_2$ where $p_1\in P_\delta$ and $p_2\in P_{\delta^\sigma}$.

We claim that $t\not\in C_H(O)$. Otherwise, $tp_1\in P_\delta$ and $tp_2\in P_{\delta^\sigma}$. Then
$$
\begin{array}{lll}
B_\sigma(p,tp)
&=B_\sigma(p_1,tp_2)+B_\sigma(p_2,tp_1),&\mbox{by \eqref{E:perps}}\\
&=B_\sigma(p_1,tp_2)+B_\sigma(tp_1,p_2),&\mbox{as $B_\sigma$ is symmetric}\\
&=0,&\mbox{as $t^\sigma=t$.}
\end{array}
$$
This contradicts our choice of\/ $p$ and thus establishes our claim.

Now $tp_1\in P_{\delta^\sigma}$ and $tp_2\in P_\delta$. So $B_\sigma(p,tp)=B_\sigma(p_1,tp_1)+B_\sigma(p_2,tp_2)$. Replace $p$ by $p_1$ or $tp_2$ so that $p\in P_\delta$ and $B_\sigma(p,tp)\ne0$. Then replace $p$ by $\sqrt{B_\sigma(p,tp)^{-1}}\,p$, so that $B_\sigma(p,tp)=1$.

Define $\ov{y}\in E_{\langle t\rangle}(P_\delta+P_{\delta^\sigma})$ by $\ov{y}(x)=B(x,tp)p+B(x,p)tp$, for all $x\in P_\delta+P_{\delta^\sigma}$. Then $\ov{y}$ is orthogonal projection onto $kp+ktp$. Moreover, $\ov{y}P_\delta\subseteq P_\delta$ and $\ov{y}P_{\delta^\sigma}\subseteq P_{\delta^\sigma}$. So $\ov{y}$ is a $\sigma$-invariant primitive idempotent in $(E(P_\delta)\times E(P_{\delta^\sigma}))^{\langle t\rangle}$.

Let $T\geq V$ such that $T/V=\langle\theta(t)\rangle$. Then $T/V$ is a complement to $N$ in $N^*$, as $N_G^*(V,Z)=N_G(V,Z)T$ and $N_G(V,Z)\cap T=V$. Now by Lemma \ref{L:onto}(ii) the restriction $\pi_{\delta,\delta^\sigma}:E_T(M)\rightarrow(E(P_\delta)\times E(P_{\delta^\sigma}))^{\langle t\rangle}$ is surjective. So by Lemma \ref{L:idempotent_lifting}(iii) there is a $\sigma$-invariant primitive idempotent $y\in E_T(M)$ such that $\pi_{\delta,\delta^\sigma}(y)=\ov{y}$. Then $yM$ is a $B$-direct summand of\/ $\op{Res}_T^GM$ which lies over $Z$ and $Z^*$. But $|T/V|=2$. So $yM\cong \op{Ind}_V^TZ$. 

Let $B_0$ be any symplectic $T$-form on $\op{Ind}_V^TZ$. Identifying $\op{Ind}_T^G(\op{Ind}_V^TZ)$ with $\op{Ind}_V^GZ$, we regard $B_0^G$ as a symplectic $G$-form on $\op{Ind}_V^GZ$. Let $\sigma_0$ be the adjoint of\/ $B_0^G$ on $E(\op{Ind}_V^GZ)$. Then $({E(P_\Delta)\!\times\! E(P_{\Delta^\sigma}}),\sigma_0)$ is an involutary $N^*$-algebra satisfying the hypothesis of Lemma \ref{L:sigma=B}(ii). So $\sigma_0$ is the adjoint of a symplectic $N^*$-form $B_{\sigma_0}$ on $P_\Delta+P_{\Delta^\sigma}$ such that $(P_\Delta)^\perp=P_\Delta$ and $(P_{\Delta^{\sigma_0}})^\perp=P_{\Delta^{\sigma_0}}$.

Let $e\in E_T(\op{Ind}_V^GZ)$ be orthogonal projection onto ${1\otimes \op{Ind}_V^TZ}$. Then ${\op{tr}}_T^G(e)=1_{\op{Ind}_V^GZ}$ and $e\in{\op{tr}}_V^T\left({E_V(\op{Ind}_V^GZ)}\,\Delta\,{E_V(\op{Ind}_V^GZ)}\right)$. Set $\ov{e}=\pi_{\Delta,\Delta^\sigma}(e)$, a primitive idempotent in $({E(P_\Delta)\!\times\!E(P_{\Delta^\sigma})})^{\langle t\rangle}$. As $O$ acts trivially on ${E(P_\Delta)\!\times\!E(P_{\Delta^\sigma}})$, we have
$$
\begin{array}{ll}
1_{P_\Delta\!+\!P_{\Delta^\sigma}}=\pi_{\Delta,\Delta^\sigma}{\op{tr}}_T^G(e)
&={\op{tr}}_{\langle \theta(t)\rangle}^{N^*}(\ov{e}),\quad\mbox{by Lemma \ref{L:onto}(ii)}\\
&=\op{tr}_1^N(\ov{e}),\quad\mbox{as $N^*=N\!:\!\langle\theta(t)\rangle$.}\\
&=\op{tr}_{O}^{C_H(O)}(\ov{e})=\op{tr}_1^{C_H(O)}(\ov{e}),\quad\mbox{as $|O|$ is odd.}\\
&={\op{tr}}_{\langle t\rangle}^{H}(\ov{e}),\quad\mbox{as $H=C_H(O):\langle t\rangle$.}\\
\end{array}
$$
Now $(\ov{e}(P_\Delta+P_{\Delta^\sigma}),B_{\sigma_0})\cong(k\langle t\rangle,B_t)$ as symmmetric $k\langle t\rangle$-modules. It then follows from Proposition \ref{P:Higman} that $(P_\Delta+P_{\Delta^\sigma},B_{\sigma_0})\cong(kHe_\eta,B_t)$.

As $B_t$ is non-degenerate on $P_\delta+P_{\delta^\sigma}$, we may choose $\ov{\omega}$ so that $\sigma_0$ is non-degenerate on
$\ov{\omega}(P_\Delta+P_{\Delta^\sigma})$. Then $\ov{\omega}$ belongs to ${E(P_\Delta)\!\times\!E(P_{\Delta^\sigma})}$, which is a $\sigma_0$-invariant semi-simple subalgebra of\/ $E(P_\Delta+P_{\Delta^\sigma})$. So Lemma \ref{L:efinS} implies that the orthogonal projection $\ov{\omega_1}$ onto $\ov{\omega}(P_\Delta+P_{\Delta^\sigma})$ belongs to $(E(P_\Delta)\times E(P_{\Delta^\sigma}))^{N^*}$.

Lemma \ref{L:onto}(ii) states that $\pi_{\Delta,\Delta^{\sigma}}:E_G(\op{Ind}_V^GZ)\rightarrow(E(P_\Delta)\times E(P_{\Delta^\sigma}))^{N^*}$ is surjective. So by Lemma \ref{L:idempotent_lifting}(iii) there is a $\sigma_0$-invariant primitive idempotent $\omega_1\in E_G(\op{Ind}_V^GZ)$ such that $\pi_{\Delta,\Delta^{\sigma}}(\omega_1)=\ov{\omega_1}$. Then $\omega_1\op{Ind}_V^GZ$ is a $B_0^G$-direct summand of\/ $\op{Ind}_V^GZ$ that is isomorphic to $M$. This completes the proof of (i) and (ii).

(iii) holds as $M$ has a $T$-projective $G$-form, but $V$ is not a symmetric vertex of\/ $M$.
\end{proof}
\end{Proposition}

\subsection{Theorems \ref{T:TleqH} and \ref{T:vertexIrreducible}}\label{SS:Proofs}

Now let $T$ be a symmetric vertex of $M$.

\begin{Lemma}\label{L:TleqH}
Let $B$ be a $T$-projective $G$-form on $M$ and let $H$ be a subgroup of $G$. Then $B$ is $H$-projective if and only if\/ $T\leq_G H$.
\end{Lemma}

\begin{proof}
The `if' implication holds by Proposition \ref{P:Higman}. The `only if' holds if\/ $T$ is a Green vertex of\/ $M$. So we assume from now on that $T$ is not a vertex of\/ $M$.

Let $V$ be a Green vertex and let $Z$ be a $V$-source of $M$. By Propositions \ref{P:case2} and \ref{P:case3} there are isometries $(\op{Ind}_V^S(Z),B_0)\rightarrow\op{Res}_S^G(M,B)$ and $(M,B_2)\rightarrow\op{Ind}_S^G(\op{Ind}_V^S(Z),B_0)$. Here $S$ is a symmetric vertex of $M$ containing $V$, $B_0$ is an $S$-form on $\op{Ind}_V^S(Z)$ and $B_2$ is a $G$-form on $M$. Now suppose that there is a $kG$-isometry $(M,B)\rightarrow\op{Ind}_H^G(L,B_1)$ for some symmetric $kH$-module $(L,B_1)$.

Composing the $3$ isometries of the previous paragraph produces a $kG$-isometry $(M,B_2)\rightarrow\op{Ind}_S^G\op{Res}_S^G\op{Ind}_H^G(L,B_1)$. Now Lemma \ref{L:Mackey} gives
$$
\op{Ind}_S^G\op{Res}_S^G\op{Ind}_H^G(L,B_1)\cong\mathop{\perp}_{g\in S\backslash G/H}\op{Ind}_{S\cap{^g\!}H}^G\op{Res}_{S\cap{^g\!}H}^{{^g\!}H}({^g\!}L,{^g\!}B_1).
$$
So by Lemma \ref{L:M->Sum} there is a $kG$-isometry
$$
(M,B_3)\rightarrow\op{Ind}_{S\cap{^g\!}H}^G\op{Res}_{S\cap{^g\!}H}^{{^g\!}H}({^g\!}L,{^g\!}B_1),
$$
for some $g\in G$ and some symmetric $G$-form $B_3$ on $M$. So $B_3$ is $S\cap{^g\!}H$-projective. But $S$ is a symmetric vertex of\/ $M$. It follows that $S\leq{^g\!}H$.

Choosing $H=T$, the work above shows that $S={^g}T$. Then taking $H$ to be any subgroup of\/ $G$, we get $T\leq_G H$.
\end{proof}

Here is a precise statement and proof of Theorem \ref{T:vertexIrreducible}:

\begin{Lemma}\label{L:vertexIrreducible}
Suppose that $M$ is a self-dual irreducible $kG$-module, with symmetric $G$-form $B$. Then the symmetric vertices of $M$ are determined up to $G$-conjugacy.

Let $V\leq T$ where $V$ is a Green vertex and $T$ a symmetric vertex of\/ $M$ and let $Z$ be a $V$-source of $M$. Then $\op{Ind}_V^TZ$ is a $B$-component of\/ $\op{Res}_T^GM$. Moreover, if $B_0$ is any $T$-form on $\op{Ind}_V^TZ$, then $M$ is a $B_0^G$-component of\/ $\op{Ind}_T^G(\op{Ind}_V^TZ)$.
\begin{proof}
Let $S$ be any symmetric vertex of\/ $M$. As $B$ is the unique $G$-form on $M$, $B$ is both $T$ and $S$-projective. Then by Theorem \ref{T:TleqH}, $T\leq_GS$ and $S\leq_GT$. So $T=_GS$. So there is only one $G$-conjugacy class of symmetric vertices of\/ $M$.

The other conclusions now follow from Propositions \ref{P:case1}, \ref{P:case2} and \ref{P:case3}.
\end{proof}
\end{Lemma}

Note that in case $T=V$ is a vertex of\/ $M$, Proposition \ref{P:case1} implies that the defect multiplicity module $P_\delta$ of\/ $E(M)$ is $P(k_N)$. But $P_\delta$ is an irreducible projective $kN$-module, by a well-known theorem of R. Kn\"orr. This forces $P_\delta=k_N$. So $V$ is a Sylow $2$-subgroup of\/ $N_G(V,Z)$.

\subsection{Real $2$-blocks}\label{SS:2blocks}

We turn our attention to the $2$-blocks of $G$. Let ${\mathcal B}$ be a real block of $kG$. So ${\mathcal B}$ is a $k(G\times G)$-direct summand of $kG$ which is $^o$-invariant. As $E_{G\times G}(kG)=Z(kG)$, we have ${\mathcal B}=kGe_{\mathcal B}$, where $e_{\mathcal B}$ is an $^o$-invariant primitive idempotent in $Z(kG)$. Let $\omega_{\mathcal B}$ be the central character of ${\mathcal B}$; the unique $k$-algebra map $Z(kG)\rightarrow k$ such that $\omega_{\mathcal B}(e_{\mathcal B})=1_k$. Recall that the conjugacy class sums form a basis for the centre $\op{Z}(kG)$ of $kG$. We will use the notation $X^+:=\sum_{x\in X}x$ for the sum of the elements of a subset $X$ of $G$, taken in $kG$. 

There are $2$-regular conjugacy classes $C_1,\dots,C_n$ of $G$ such that $e_{\mathcal B}=\sum_{i=1}^n\alpha_iC_i^+$, with non-zero $\alpha_i\in k$.
For each $i$ let $\ov{i}$ be the index of the class of inverses of the elements of $C_i$. Then $\alpha_{\ov{i}}=\alpha_i$, as ${\mathcal B}$ is real. Choose $c_i\in C_i$, such that $c_{\ov{i}}=c_i^{-1}$ if $C_{\ov{i}}\ne C_i$. Set $C_G(c_i)$ and $C_G^*(c_i)=N_G(\{c_i,c_i^{-1}\})$ as the centralizer and extended centralizer of $c_i$ in $G$, respectively.

Let $E_i$ be a Sylow $2$-subgroup of $C_G^*(c_i)$ and let $D_i=C_{E_i}(c_i)$. So $[E_i:D_i]\leq2$, with equality if and only if $C_i$ is a non-trivial real class. We call $D_i$ a defect group of $C_i$ and $E_i$ an extended defect group of $C_i$.
When $C_i$ and $C_j$ are real classes, we write $(D_i,E_i)\leq_G(D_j,E_j)$ if $D_i^g\leq D_j$ and $E_j=D_jE_i^g$, for some $g\in G$.

R. Brauer showed that $\alpha_j\ne0_k$ and $\omega_{\mathcal B}(C_j^+)\ne0_K$, for some $j$. Any such $C_j$ is a defect class of ${\mathcal B}$. Then $D_j$ is independent of the choice of $C_j$, and is called a defect group of ${\mathcal B}$. Moreover $D_i\leq_G D_j$, for all $C_i$.

R. Gow showed that there is a real defect class $C_j$. Then $E_j$ is independent of the choice of real $C_j$ \cite[Theorem 2.1]{GowReal2Blocks}. Gow called $E_j$ an extended defect group of ${\mathcal B}$. In \cite[Proposition 14]{MurrayStrong} we proved that $(D_i,E_i)\leq_G(D_j,E_j)$, for all real $C_i$. In particular, we can and do choose $c_i$ so that $D_i\leq D_j$, and moreover $E_j=D_jE_i$, if $C_i$ is real. So for real $C_i$ we have $c_i^{e_i}=c_i^{-1}$ for some $e_i\in E_j\backslash D_j$.

Now consider the involutary $G$-algebra $(kG,^o)$. Each subgroup $H$ of $G$ defines a trace map $\op{tr}_H^G:kG^H\rightarrow Z(kG)$. Then for all $i$
$$
C_i^+=\op{tr}_{C_G(c_i)}^G(c_i)=\op{tr}_{D_i}^G(c_i),\quad\mbox{as $[C_G(c_i):D_i]$ is odd.}
$$
Set $\theta=\sum_{i=1}^n\alpha_i\op{tr}_{D_i}^{E_j}(c_i)$. As $C_i^+=\op{tr}_{E_j}^G(\op{tr}_{D_i}^{E_j}(c_i))$ and $\op{tr}_{D_{\ov{i}}}^{E_j}(c_{\ov{i}})=\op{tr}_{D_i}^{E_j}(c_i)^o$ and $\alpha_{\ov{i}}=\alpha_i$, we have
$$
\theta\in kG^{E_j},\quad\theta^o=\theta\quad\mbox{and}\quad e_{\mathcal B}=\op{tr}_{E_j}^G(\theta).
$$
We use this to prove parts (i) and (ii) of Theorem \ref{T:vertexBlock}:

\begin{Lemma}\label{L:part(i)}
Each symmetric vertex of an indecomposable ${\mathcal B}$-module is contained in an extended defect group of ${\mathcal B}$.
\begin{proof}
Let $(M,B)$ be a symmetric ${\mathcal B}$-module, where $M$ is indecomposable, and let $\sigma$ be the adjoint of $B$ on $E(M)$. Let $\pi:kG\rightarrow E(M)$ define the module structure on $M$. Then $\pi$ maps $kG^H$ into $E_H(M)$, for all $H\leq G$. Moreover $\pi(g)^\sigma=\pi(g^{-1})$, for all $g\in G$, as $B$ is $G$-invariant. So $\pi(\theta)$ is a $\sigma$-invariant element of $E_{E_j}(M)$ and
$$
\op{tr}_{E_j}^G(\pi(\theta))=\pi(\op{tr}_{E_j}^G(\theta))=\pi(e_{\mathcal B})=1_M.
$$
We conclude from Proposition \ref{P:Higman} that $B$ is $E_j$-projective.
\end{proof}
\end{Lemma}

\begin{Corollary}\label{C:part(ii)}
There is a self-dual irreducible ${\mathcal B}$-module whose symmetric vertices are the extended defect groups of ${\mathcal B}$.
\begin{proof}
If ${\mathcal B}$ is the principal $2$-block of $G$, the symmetric vertices of the trivial $kG$-module are the Sylow $2$-subgroups of $G$. These are also the extended defect groups of ${\mathcal B}$. Suppose then that ${\mathcal B}$ is not the principal block. Now \cite[Proposition 1.4(v)]{GowWillemsPIMs} implies that the defect group $D_j$ of ${\mathcal B}$ is a Green vertex of some self-dual irreducible ${\mathcal B}$-module $M$. Let $T$ be a symmetric vertex of $M$ containing $D_j$. Then $[T:D_j]=2$, by Theorem \ref{T:vertex}. But $T\leq_G E_j$, according to Lemma \ref{L:part(i)}. Since $[E_j:D_j]=2$ we conclude that $T=_G E_j$ is an extended defect group of ${\mathcal B}$.
\end{proof}
\end{Corollary}

To prove part (iii) of Theorem \ref{T:vertexBlock}, we  work in the involutary $G\times G$-algebra $(E(kG),\sigma)$. Here $\sigma$ is the adjoint of $B_1$, as in Section \ref{SS:Projective}. Recall the identification \eqref{E:E(kG)} of $x\otimes y$ as an endomorphism of $kG$, for $x,y\in kG$. Then $(g_1,g_2)\cdot x\otimes y=g_1xg_2^{-1}\otimes g_1yg_2^{-1}$, for all $g_1,g_2\in G$, as can be checked. So
$$
\op{Stab}_{G\times G}(x\otimes y)=\{(g_1,g_2)\in G\times G\mid g_1x=xg_2, g_1y=yg_2\}.
$$
Now $C_G(x)=C_G(x^2)$, if $x\in G$ has odd order. It follows from this that 
$$
\op{Stab}_{G\times G}(x\otimes x^{-1})=\Delta C_G(x),\quad\mbox{if $x\in G$ is $2$-regular.}
$$
Set $\Delta H=\{(h,h)\mid h\in H\}$, for $H\leq G$. Let $D$ be a defect group of $x$ and let $C$ be the conjugacy class of $G$ containing $x^2$. Then $\op{tr}_{\Delta D}^{G\times G}(x\otimes x^{-1})=C^+$.

Recall that $C_j$ is a real defect class of ${\mathcal B}$, $D_j$ is a defect group of ${\mathcal B}$ and $E_j$ is an extended defect group of ${\mathcal B}$. Moreover each class $C_i$ occurring in $e_{\mathcal B}$ is $2$-regular. So there is a unique $d_i\in\langle c_i\rangle$ such that $d_i^2=c_i$. Define
$$
\Theta=\sum_{i=1}^n\alpha_i\op{tr}_{\Delta D_i}^{\Delta E_j}(d_i\otimes d_i^{-1}).
$$
If $c_i$ is real, then $c_i^{e_i}=c_i^{-1}$, where $e_i\in E_j\backslash D_j$. So $e_id_ie_i^{-1}=d_i^{-1}$ and
$$
(e_i,e_i)\cdot d_i\otimes d_i^{-1}=d_i^{-1}\otimes d_i=(d_i\otimes d_i^{-1})^\sigma,\quad\mbox{using \eqref{E:sigma}.}
$$
It follows from this that $\op{tr}_{\Delta D_{\ov{i}}}^{\Delta E_j}(d_{\ov{i}}\otimes d_{\ov{i}}^{-1})=\op{tr}_{\Delta D_i}^{\Delta E_j}(d_i\otimes d_i^{-1})^\sigma$, for all $i$. So
\begin{equation}\label{E:Theta}
\Theta^\sigma=\Theta\quad\mbox{and}\quad\op{tr}_{\Delta E_j}^{G\times G}(\Theta)=e_{\mathcal B}.
\end{equation}
Now notice that $e_{\mathcal B}$ is the identity element of $E_{G\times G}({\mathcal B})=Z({\mathcal B})$.

\begin{Lemma}\label{L:part(iii)}
Let $E$ be an extended defect group of\/ ${\mathcal B}$. Then the restriction of\/ $B_1$ to ${\mathcal B}$ is non-degenerate and $\Delta E$-projective. So $\Delta E$ is a symmetric vertex of\/ ${\mathcal B}$.
\begin{proof}
Let $D\leq E$ be a defect group of ${\mathcal B}$. So $\Delta D$ is a Green vertex of ${\mathcal B}$ as $k(G\times G)$-module. Now ${\mathcal B}$ is the only direct summand of $kG$ isomorphic to its $k(G\times G)$-dual ${\mathcal B}^*$. So $B_1$ is non-degenerate on ${\mathcal B}$, by Lemma \ref{L:GowWillems}. Then \eqref{E:Theta} and Proposition \ref{P:Higman} imply that $B_1$ is $\Delta E$-projective.

Suppose that ${\mathcal B}$ is the principal 2-block of $G$. Then $E=D$ is a defect group of ${\mathcal B}$. So $\Delta E$ is , as $k(G\times G)$-module. As $B_1$ is $\Delta E$-projective, Proposition \ref{P:case1} implies that $\Delta E$ is a symmetric vertex of ${\mathcal B}$.

Suppose that ${\mathcal B}$ is not the principal 2-block of $G$. Then ${\mathcal B}$ belongs to a non-principal $2$-block of $G\times G$. So $\Delta D$ is not a symmetric vertex of ${\mathcal B}$, by Corollary \ref{C:principal}. As $B_1$ is $\Delta E$-projective and $[\Delta E:\Delta D]=2$, Proposition \ref{P:case2} implies that $\Delta E$ is a symmetric vertex of ${\mathcal B}$.
\end{proof}
\end{Lemma}

\subsection{Examples}\label{SS:Examples}

We give examples to show that all cases in Theorem \ref{T:vertex} actually occur. So as above let $M$ be a self-dual indecomposable $kG$-module of symmetric type, let $T$ be a symmetric vertex of $M$, let $V$ be a Green-vertex of $M$ contained in $T$ and let $Z$ be a $V$-source of $M$. In view of the literature on vertices, most of our examples involve a self-dual irreducible $kG$-module.

To get examples of Theorem \ref{T:vertex}(i), choose $M$ so that $V$ is a Sylow $2$-subgroup of $G$. Then the symmetric vertex $T$ will also be a Sylow $2$-subgroup of $G$. For a non-trivial example, if $S_{2n}$ is the symmetric group of degree $2n\geq6$ and $M$ is the irreducible $kS_{2n}$-module labelled by the $2$-regular partition $[2n-1,1]$, then $V$ is a Sylow $2$-subgroup of $S_{2n}$ according to \cite{MZ}.

There are examples of Theorem \ref{T:vertex}(i) where $V$ is not a Sylow $2$-subgroup of $G$: let $M$ be the $4$-dimensional irreducible $kS_5$-module labelled by $[3,2]$. Then according to \cite{DKZ} $V$ is conjugate to a Klein-four subgroup of $A_4\leq S_5$ and $\op{Res}_V^{S_5}(M)=Z_1\oplus Z_2$ is the direct sum of two non-isomorphic self-dual $V$-sources. Then it follows from Lemma \ref{L:GowWillems} and Proposition \ref{P:case1} that $V=T$ is a symmetric vertex of $M$.

Now suppose that $G$ has even order and $M$ is irreducible and projective as $kG$-module. Then $M$ belongs to a real $2$-block ${\mathcal B}$ of $G$ which has a trivial defect group $V=1$. According to \cite{GowReal2Blocks}, there is an involution $t\in G$ such that $\langle t\rangle$ is an extended defect group of ${\mathcal B}$. It then follows from Lemma \ref{L:GreenSymmetric} that $\langle t\rangle$ is a symmetric vertex of $M$. As $Z=k$ is self-dual, this gives an example of Theorem \ref{T:vertex}(ii).

Next let $M$ be the irreducible $kS_7$-module labelled by $[4,3]$. It is shown in \cite{DKZ} that $M$ shares its vertices and sources with the irreducible $kS_5$-module labelled by $[3,2]$, as discussed above. Once again we have an example of Theorem \ref{T:vertex}(ii), but here the vertex $V$ is not a defect group of the $2$-block ${\mathcal B}$.

It is relatively difficult to find examples of Theorem \ref{T:vertex}(iii). Take $n\geq3$, and set $H=\op{GL}(n,2)$ and $G=H:\langle\tau\rangle$, where $\tau$ is the transpose inverse automorphism of $H$. Let $M_1$ be the natural $n$-dimensional $kH$-module and set $M=\op{Ind}_H^GM_1$. Then $V$ is a Sylow $2$-subgroup of $H$ and $Z=\op{Res}_V^H(M_1)$ is a $V$-source of $M$. It is not too difficult to check that $Z$ is not self-dual, giving us an infinite family of examples of Theorem \ref{T:vertex}(iii). We note that in case $n=4$, $H\cong A_8$ and $G\cong S_8$. Then we can identify $M$ with the irreducible $kS_8$-module labelled by $[5,3]$.

E. C. Dade \cite{Dade} gave us an example of a solvable group which has a self-dual irreducible module whose sources are not self-dual: Let $G$ be the semi-direct product of an extra-special group of order $27$ and exponent $3$ with a semi-dihedral group of order $16$. Take $M$ to be the unique $6$-dimensional $kG$-module. Then $V$ is a quaternion group of order $8$ and $\op{Res}_V^GM=Z\oplus Z^*$ with $Z\not\cong Z^*$. We describe this in more detail in \cite{MurrayNavarro}.

In \cite{MurrayNavarro} the author and G. Navarro use the results of the current paper to show that if $G$ is solvable, and $M$ is irreducible, then case (i) of Theorem \ref{T:vertex} does not occur. Moreover a symmetric vertex of $M$ splits over a Green vertex if and only if $P(M)$ is of quadratic type, in the sense of \cite{GowWillemsPIMs}. 

\section{Acknowledgement}

Many years ago Burkhard K\"ulshammer challenged me to generalise the notion of extended defect groups from real $2$-blocks to other self-dual objects. This inspired me to write this paper. Discussions with Adam Allan, Everett C. Dade, Rod Gow and J\"urgen M\"uller helped to clarify my ideas. I would also like to thank the anonymous referees for picking up a number of errors, in particular a false conclusion in an earlier version of Lemma \ref{L:idempotent_lifting}.


\end{document}